\documentclass[11pt]{article}

\usepackage{hyperref}
\usepackage{blkarray}
\usepackage{amsmath}
\usepackage{amssymb}
\usepackage{amsfonts}
\usepackage{amsthm}
\usepackage{enumerate}
\usepackage{multicol}
\usepackage{colortbl}

\usepackage{pgf}
\usepackage{tikz}
\usepackage{xcolor}
\usetikzlibrary{arrows}
\usetikzlibrary{patterns}
\usetikzlibrary{petri}

\usepackage[english]{babel}
\usepackage{tabularx}
\usepackage{subfig}

\newcommand{\rank}{{\rm rank}}
\newcommand{\p}{p}%{{\bf p}}

\newcommand{\bu}{{u}}
\newcommand{\bm}{m}%{{\bf m}} 

\newcommand{\C}{\mathcal C}

\newcommand{\gs}{\mathcal G_{\C}} %gain space

\newcommand{\pogp}{\langle G', \bm' \rangle}

\newcommand{\pogo}{\langle \overline G, \overline \bm \rangle}
\newcommand{\pogom}{\langle G, \overline \bm \rangle}

\newcommand{\tgT}{\langle G, \bm_T \rangle}

\newcommand{\Tor}{\mathcal T_0} %Fixed Torus
\newcommand{\T}{\mathcal T} %Flex Torus d-dimensions

\newcommand{\pofw}{(\langle G, \bm \rangle,\p)} %Periodic Orbit FrameWork
\newcommand{\pofwpr}{(\langle G', \bm' \rangle,\p')} %Periodic Orbit FrameWork Affine

\newcommand{\pog}{\langle G, \bm \rangle} %Periodic Orbit Graph
\newcommand{\dpfwo}{( \langle G^{\bm}, L_0 \rangle, \p^{\bm})} %Derived Periodic FrameWork fix
\newcommand{\bbog}{\langle H, \bm \rangle} %Bar Body Orbit Graph
\newcommand{\R}{\mathbf R} %Rigidity Matrix

\definecolor{desk}{rgb}{.345, .306, .216}
\definecolor{vancouver}{rgb}{.412, .412,.412}
\definecolor{beetle}{rgb}{.180, .161, .102}
\definecolor{bluey}{rgb}{.235, .380, .415}
\definecolor{melon}{rgb}{1, .259, .259}
\definecolor{vneck}{rgb}{.596, .282, .376}
\definecolor{pink}{rgb}{.918, .122, .545}
\definecolor{mango}{rgb}{1, .8, .267}
\definecolor{lips}{rgb}{.541, .074, .239}
\definecolor{sage}{rgb}{.522, .604, .247}
\definecolor{moss}{rgb}{.184, .224, .129}
\definecolor{cumin}{rgb}{.6, .580, 0}
\definecolor{lichen}{rgb}{.745, .998, .729}
\definecolor{rain}{rgb}{.780, .812, .706 }
\definecolor{cloud}{rgb}{.961, .976, .870}
\definecolor{couch}{rgb}{.8, 1, .2}
\definecolor{cement}{rgb}{.678, .682, .549}
\definecolor{sky}{rgb}{.278, .514, 1}

\newtheorem{thm}{Theorem}[section]
\newtheorem{cor}[thm]{Corollary}
\newtheorem{lem}[thm]{Lemma}
\newtheorem{prop}[thm]{Proposition}

\theoremstyle{remark}

\theoremstyle{remark}
\newtheorem{ex}[thm]{Example}
\theoremstyle{remark}

\begin{document}

\title{Inductive constructions for frameworks on a two-dimensional fixed torus}
\author{
{Elissa  Ross
\thanks{eross2@wpi.edu, Worcester Polytechnic Institute, Worcester, Masschusetts}}
}
\maketitle

\begin{abstract} 
%\ER{In this paper we find necessary and sufficient conditions for the minimal rigidity of graphs on the two-dimensional fixed torus, $\Tor^2$. %These results can be framed in relation to two well known results in the rigidity of finite graphs, namely Henneberg's Theorem (Theorem \ref{thm:HennebergFinite}) and Laman's Theorem (Theorem \ref{thm:LamanFinite}). 
%We use these periodic orbit frameworks (gain graphs) as models of infinite periodic graphs, and the rigidity of the gain graphs on the torus correspond to the generic rigidity of the periodic framework under forced periodicity. 
%Here it is shown that every minimally rigid periodic orbit framework on $\Tor^2$ can be constructed from smaller graphs through a series of inductive constructions. This is a periodic version of Henneberg's theorem about finite graphs. We also describe a characterization of the generic rigidity of a two-dimensional periodic framework through a consideration of the gain assignment on the corresponding periodic orbit framework. This can be viewed as a periodic analogue of Laman's theorem about finite graphs.} \\	

An infinite periodic framework  in the plane can be represented as a framework on a torus, using a $\mathbb Z^2$-labelled gain graph. We find necessary and sufficient conditions for the generic minimal rigidity of frameworks on the two-dimensional fixed torus $\Tor^2$. It is also shown that every minimally rigid periodic orbit framework on $\Tor^2$ can be constructed from smaller frameworks through a series of inductive constructions. These are fixed torus adapted versions of the results of Laman and Henneberg respectively for finite frameworks in the plane. The proofs involve the development of inductive constructions for $\mathbb Z^2$-labelled graphs. \\

%rigidity and flexibility of structures

\end{abstract}

%\tableofcontents

%%%%INTRODUCTION
\section{Introduction}
The study of the rigidity of periodic frameworks has witnessed an explosion of interest in recent years \cite{BorceaStreinuII,periodicFrameworksAndFlexibility,DeterminancyRepetitive,Theran,InfiniteBarJointFrameworksCrystals,PowerAffine,SymmetryPeriodic}. This is due in part to questions raised by the materials science community about the rigidity or flexibility of {\it zeolites}, a type of mineral with crystalline structure characterized by a repetitive porous pattern \cite{flexibilityWindow, enumerationTetrahedral}. This type of material can be modelled as a fragment of an infinite periodic framework. 

A periodic framework in the plane consists of a locally finite infinite graph together with a placement of the vertices in $\mathbb R^2$ such that the resulting object is symmetric with respect to the free action of $\mathbb Z^2$ \cite{periodicFrameworksAndFlexibility}. Such a periodic framework has a finite number of vertex and edge orbits with respect to the periodic (translational) symmetry. 

%\ER{here is some stuff that was cut out later, including an important reference:}
%There are essential differences between the rigidity of finite and infinite frameworks (namely an infinite dimensional rigidity matrix), which eliminates the key relationships between independence and infinitesimal rigidity for infinite graphs \cite{DeterminancyRepetitive}. 
%\ER{end of pasted stuff}
%As a result, new approaches are required to study the rigidity of periodic frameworks. 

In this paper we consider periodic frameworks as orbit frameworks on a topological torus, where we use the torus as a ``fundamental region" for a tiling of the plane. In particular, we study frameworks on a flat torus of fixed size and shape (the {\it fixed torus}). %\ER{probably want to mention forced periodicity here}
Considering the rigidity of frameworks on the fixed torus corresponds to the study of the rigidity of frameworks with {\it forced periodicity}. Such a framework is constrained to remain periodic (with the same translational symmetry) throughout its motion.  
%Experimental evidence has suggested that for some molecular compounds, the time scales of lattice movement (e.g. deformations of the torus) may be several orders of magnitude slower than the molecular deformations within the lattice, making the ``fixed torus" relevant for the study of rigidity \cite{ThorpePrivate}. 
When the torus is {\it not} fixed, and we instead allow the torus to undergo affine deformation, then a flexible framework on this {\it variable torus} corresponds to an infinite periodic framework with the property that the velocities of the vertices that are ``far away from the centre" will become arbitrarily large. This may be problematic for the representation of physical materials as periodic frameworks, which partly motivates the use of the fixed torus. 

We regard an infinite periodic framework as a finite graph $G$ realized on the fixed torus, where $G$ is the orbit graph under the periodic symmetry. To represent this, we use a labelled multigraph $\pog$ where $m$ is a labelling of the directed edges of $G$ by elements of $\mathbb Z^2$. This is a {\it gain graph:} a graph with edges that are labeled invertibly by group elements. The edge labels provide information about how the framework `wraps' around the torus, or equivalently, how the periodic framework is connected together. The gains can also be summed along paths or cycles in the graph, which will be crucial for providing a combinatorial description of generic rigidity on the fixed torus.

In a previous paper \cite{ThesisPaper1} we found necessary conditions for the rigidity of a framework on a fixed $d$-dimensional torus. The present paper provides sufficient conditions for the generic rigidity of a framework $\pog$ on a $2$-dimensional fixed torus, which depends in part on the labelling $m$ of the edges. In particular, we prove that the generic rigidity of a periodic orbit framework on the torus depends on what we call {\it constructive cycles}, which are unbalanced cycles in the gain graph (the sum of the labels on the cycle is not the group identity).
%\begin{thmlam}
%Let $\pog$ be a periodic orbit graph. Then $\pog$ is generically minimally rigid on the fixed torus if and only if $\pog$ satisfies 
%\begin{enumerate}[(i)]
%\item $|E| = 2|V| - 2$, and $|E'| \leq 2|V'| - 2$ for all subgraphs $G' \subset G$
%\item $\bm$ is a constructive gain assignment. 
%\end{enumerate}
%\end{thmlam}
This can be viewed as a fixed torus version of Laman's theorem, which characterizes the generic rigidity of finite frameworks in the plane. This result was proved independently in \cite{Theran}, as part of a more general statement about frameworks with a variable lattice (i.e. frameworks on a variable torus), which we elaborate on in the next section. 

We also prove a fixed torus adapted version of Henneberg's theorem, which is a recursive method of creating larger generically rigid frameworks from smaller ones.
%\begin{thmhen}
%A periodic orbit graph $\pog$ on the fixed torus is generically minimally rigid if and only if it can be constructed from a single vertex on the fixed torus by a sequence of periodic Henneberg moves. 
%\end{thmhen}
To prove both results, we use {\it inductive techniques}, building up minimally rigid frameworks from smaller minimally rigid frameworks by systematically adding vertices and edges to the underlying graph according to certain rules. These inductive moves add to a vocabulary of methods that may be applied to a broad class of problems concerning periodic frameworks. The inductive techniques presented here are truly ``local" moves, in the sense that they could be viewed as usual Henneberg moves from finite rigidity theory performed on each cell of a periodic framework simultaneously (see Figures \ref{fig:vertexAdditionTorus} and \ref{fig:edgeSplitTorus}). In particular, the torus moves are defined to preserve the net gains on any cycles they alter, which in turn preserves the basic structure of the periodic framework. Defining the inductive moves this way allows us to prove that they preserve infinitesimal rigidity using the basic ideas for finite frameworks due to Whiteley \cite{SomeMatroids}. In addition, the proof of Laman's theorem found in \cite{SomeMatroids} can be easily adapted to our setting, demonstrating that proving rigidity results for group-labeled graphs is no harder than for unlabelled graphs in this case.

\subsection{Results in context}

One of the earliest investigations into frameworks on a fixed torus is due to Whiteley \cite{UnionMatroids}.  The following result from that paper shows that given a graph $G$ with certain combinatorial properties, we can always find an appropriate gain assignment $\bm$ and geometric realization $\p$ to yield a minimally rigid framework on the $d$-dimensional fixed torus (the equivalence of {\it (i)} and {\it (ii)} is a well known result due to Nash-Williams \cite{NashWilliams}). 

%Walter's theorem that we build on
\begin{thm}[Whiteley, \cite{UnionMatroids}]
For a multigraph $G$, the following are equivalent:
\begin{enumerate}[(i)]
\item $G$ satisfies $|E| = d|V|-d$, and every subgraph $G' \subseteq G$ satisfies $|E'| \leq d|V'| - d$, 
\item $G$ is the union of $d$ edge-disjoint spanning trees,
\item For some gain assignment $m$ and some realization $p$, the framework $\pofw$  is minimally  rigid on the $d$-dimensional fixed torus.
\end{enumerate}
\label{waltersThm}
\end{thm}
\noindent The goal of the present paper is to strengthen and broaden the scope of Theorem \ref{waltersThm} in two dimensions. In particular, we will answer the question ``for {\it what} gain assignments $m$ is $\pog$ generically minimally rigid on $\Tor^2$?"

The algebraic-geometric theory of $d$-dimensional periodic frameworks was set out by Borcea and Streinu \cite{periodicFrameworksAndFlexibility}. In our language their set-up corresponds to the {\it variable torus}, where the periodic lattice is allowed to undergo affine transformation. The basic ideas of \cite{periodicFrameworksAndFlexibility} have now become standard for the study of periodic frameworks. The ideas presented in this paper are specialized to the fixed torus, and were developed independently in \cite{myThesis}.

As mentioned above, Malestein and Theran have established a characterization of generic rigidity on the variable torus \cite{Theran}. The broad range of methods used there are quite different from what is presented here. They involve matroid representations and the introduction of {\it periodic direction networks}, which are a type of infinite multigraph where each edge is (periodically) assigned a direction. It is not clear that inductive methods such as those presented here can be used to characterize rigidity on a variable torus.   The development of inductive techniques for the variable torus is a challenging and interesting open problem (see also Section \ref{sec:conclusion} for further discussion). 

Inductive techniques are both general and widely used \cite{inductionSurvey}. In addition to providing a characterization of generic rigidity in the plane \cite{Henneberg}, Henneberg-type moves easily adapt to $d$-dimensional (finite) frameworks. Inductive techniques have also played a key role in the development of {\it global rigidity}, the study of graphs with unique realizations \cite{bergJordan, GenGlobalRigidity}. Furthermore, inductive methods also appear in the study of special classes of frameworks, for example Schulze's work on symmetric frameworks \cite{BS3}, and Nixon, Owen and Power's exploration of frameworks supported on surfaces embedded in $\mathbb R^3$ \cite{Nixon}. The inductive methods presented in the paper may be useful to prove other results about periodic frameworks.

\subsection{Outline of paper}
In Section \ref{sec:background} we review the basic ideas about periodic frameworks and their representations as orbit frameworks on a fixed torus. We also outline the key facts about the rigidity of these frameworks. In Section \ref{sec:generatingIsostaticFrameworks} we define gain-preserving inductive constructions on periodic orbit frameworks, and prove our first main result, namely a fixed torus Henneberg theorem. Section \ref{sec:combinatorial} describes the combinatorial structure of the class of graphs which are minimally rigid on the fixed torus. Building on those results, in Section \ref{sec:gainAssignmentsDetermineRigidity} we prove the second main theorem, a Laman-type theorem, which characterizes minimal rigidity for frameworks on the fixed torus. Finally, Section \ref{sec:conclusion} closes by connecting this work with some existing extensions, and identifying some areas for further work. 

%%%%%%%%%%%%%%
\section{Background}
\label{sec:background}
The full background for the present work is recorded in an earlier paper \cite{ThesisPaper1}. We summarize here the essential definitions and results.

\subsection{Graph theory conventions}
\label{sec:graphTheory}

We denote a graph by $G=(V, E)$ where $V = V(G)$ and $E=E(G)$ are the vertex and edge sets. We assume that $G$ is a multigraph, with multiple edges permitted. To simplify notation, if a graph is denoted by $G_k$, we use $V_k$ and $E_k$ to denote the vertex and edge sets respectively.  

A {\it subgraph} $H \subseteq G$ is a graph whose vertex set is a subset of that of $G$. A subgraph is called {\it vertex-induced} is for every pair of vertices $x y \in V(H)$, the edge connecting $x$ and $y$ is an element of $E(H)$ if and only if it is an edge of $G$. In this case $E(H)$ is called the edge set {\it spanned by} $V(H)$. 

For a vertex $x$ in a graph $G$, the {\it neighbours} of $x$ are the vertices of $G$ that are connected to $x$ by an edge, and we denote this by set by $N(x)$. Two vertices are called {\it adjacent} if they are connected by an edge. Two edges are {\it incident} if they share a vertex. An edge and a vertex on that edge are also called {\it incident}. The following terminology from \cite{pebbleGameSparse} will be useful. A graph $G$ is $(k, \ell)${\it -sparse} if every subset of vertices $V' \subseteq V$ satisfies $|E'| \subseteq k|V'| - \ell$, where $E'$ is the set of edges spanned by $V'$. If, in addition, $|E| = k|V| - \ell$, then $G$ is called $(k, \ell)${\it -tight}.

%%%%%%%%%%%%% 
\subsection{Periodic orbit frameworks on $\Tor^2$}

{\it Periodic frameworks} in the plane are locally finite infinite graphs which are symmetric with respect to the free action of $\mathbb Z^2$, together with a periodic realization of the vertices in $\mathbb R^2$. This implies that the framework has a finite number of vertex and edge orbits under the action of $\mathbb Z^2$.  Further details on periodic frameworks can be found in the work of Borcea and Streinu \cite{BorceaStreinuII,periodicFrameworksAndFlexibility}. 

Our approach in this paper is to consider periodic frameworks as orbit frameworks on a torus. The $2$-dimensional topological torus can naturally be considered a {\it fundamental region} for a tiling of the plane.  We will consider the rigidity of frameworks on the torus as a model of the rigidity of periodic frameworks in the plane. 

In particular, we will consider frameworks on the {\it fixed torus}, which we now define. Let $L_0$ be a $2 \times 2$ matrix whose rows are independent vectors of the form $(x, 0), (y_1, y_2)$, $x, y_1, y_2 \in \mathbb R$. Let $ \mathbb Z^2 L_0$ denote the group generated by the rows of $L_0$, viewed as translations of $\mathbb R^2$. We call $\mathbb Z^2 L_0$ the {\it fixed lattice}, and $L_0$ the {\it lattice matrix}. The quotient space $\mathbb R^2 / \mathbb Z^2 L_0 $ will be called the {\it fixed torus}, and we denote it by $\Tor^2$. 

%GAIN GRAPHS
\subsubsection{Gain graphs}
A {\it gain graph} $\pog$ is a directed multigraph $G=(V, E)$ together with an invertible labelling of the edges by the elements of a group $\mathcal G$, which is called the {\it gain group}. In particular, $m:E^+ \rightarrow \mathcal G$ where $E^+$ represents the forward-directed edges of $E$. If a directed edge $e \in E$ has $m(e) = g$, then the other direction of the edge ($-e$) has label $g^{-1}$. The group label on the edge is called the {\it gain}. In the context of graphs embedded into surfaces, gain graphs are often called {\it voltage graphs} \cite{TopologicalGraphTheory, BiasedGraphsI}. 

If $\pog$ is a gain graph, and $G' \subset G$ is a subgraph of $G$, then $G'$ induces a sub-(gain)-graph of $\pog$, which we denote by $\langle G', m' \rangle$. Here $m'$ is the labelling on the edges from $\pog$, restricted to the edges of $G'$. %\ER{Going to need to adopt this throughout.}

%POGS
\subsubsection{Periodic orbit graphs and frameworks}
A (two-dimensional) {\it periodic orbit framework} is a pair $\pofw$, where $\pog$ is a gain graph with the two-dimensional integer lattice $\mathbb Z^2$ as the gain group, together with a map $p: V \rightarrow \Tor^2$ that describes the position of the vertices of $\pog$ on the fixed torus $\Tor^2$. The graph $\pog$ will be called a {\it periodic orbit graph}. 

The edges of $\pog$ (denoted $E\pog$) are recorded as follows: $e=\{v_i, v_j; m_e\}$, where $m_e \in \mathbb Z^2$. Since $m$ labels the edges of $G$ invertibly, it follows that we can equivalently write $e=\{v_j, v_i; -m_e\}$.

An example of a periodic orbit graph $\pog$ is shown in Figure \ref{fig:gainGraph}(a). From the periodic orbit graph $\pog$ we can define the {\it derived graph} $G^m$ (also called the {\it covering graph} \cite{TopologicalGraphTheory, gainGraphBibliography}), shown in Figure \ref{fig:gainGraph}(b). The derived graph $G^m$ has vertex set $V^m$ and $E^m$ where $V^{\bm}$ is the Cartesian product $V \times \mathbb Z^2$, and $E^{\bm} = E \times \mathbb Z^2$. Vertices of $V^{\bm}$ have the form  $(v_i, a)$, where $v_i \in V$, and $a \in \mathbb Z^2$. Edges of $E^{\bm}$ are denoted similarly.  If $e$ is the directed edge connecting vertex $v_i$ to $v_j$ in $\pog$, and $b$ is the gain assigned to the edge $e$, then the edge $(e, a) = \{(v_i,a), (v_j, a+b)\}$ of $G^{\bm}$ connects vertex $(v_i,a)$ to $(v_j, a+b)$. Thus, the derived graph is a graph whose automorphism group contains $\mathbb Z^2$. %Furthermore, the periodic orbit graph $\pog$ can be seen as a kind of `recipe' for the infinite derived graph $G^m$. \ER{More precisely, the derived graph $G^m$ is the covering space of $\pog$ \cite{TopologicalGraphTheory}. }

In a similar way, from the periodic orbit framework $\pofw$ we can define the {\it derived periodic framework} $\dpfwo$, where $p^m$  is given by
\[p^m(v, z) = p(v) + zL_0.\]

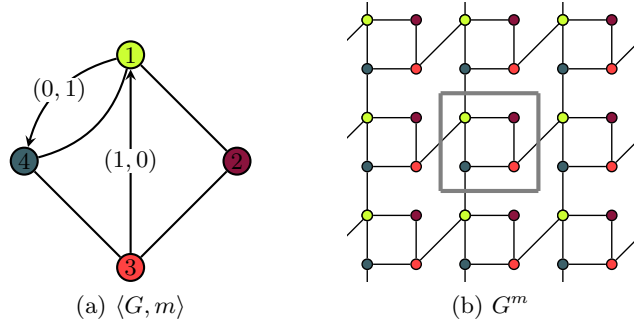
\begin{figure}[h!]
\begin{center}
\subfloat[$\pog$]{\label{fig:finiteGraph}\begin{tikzpicture}[auto, node distance=2cm, thick]
\tikzstyle{vertex1}=[circle, draw, fill=couch, inner sep=1pt, minimum width=3pt, font=\footnotesize];
\tikzstyle{vertex2}=[circle, draw, fill=lips, inner sep=1pt, minimum width=3pt, font=\footnotesize];
\tikzstyle{vertex3}=[circle, draw, fill=melon, inner sep=1pt, minimum width=3pt, font=\footnotesize];
\tikzstyle{vertex4}=[circle, draw, fill=bluey, inner sep=1pt, minimum width=3pt, font=\footnotesize];
\tikzstyle{gain} = [fill=white, inner sep = 0pt,  font=\footnotesize, anchor=center];

\node[vertex1] (1) {$1$};
\node[vertex2] (2) [below right of=1] {$2$};
\node[vertex3] (3) [below left of=2] {$3$};
\node[vertex4] (4) [below left of=1] {$4$};

%\path[thick] (1) edge (2)
%(2) edge (3)
%(3) edge   (1)
%(1) edge [bend right]  (4)
%(4) edge [bend right]   (1)
%(4) edge  (3);

\pgfsetarrowsend{stealth}[ shorten >=1cm]
\path
(3) edge node[gain] {$(1,0)$} (1)
(1) edge [bend right] node[gain] {$(0,1)$} (4);
\pgfsetarrowsend{}

%edges with (0,0) gain
\path (1) edge  (2)
(2) edge  (3)
(4) edge [bend right]   (1)
(4) edge  (3);

\end{tikzpicture}}\hspace{.5in}%
\subfloat[$G^m$]{\label{fig:periodicFramework}\includegraphics[width=1.5in]{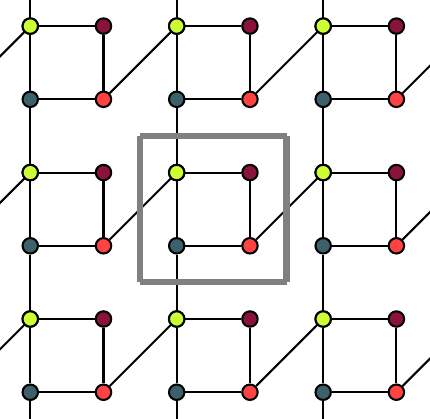}}

\caption{A periodic orbit graph $\pog$, where $m:E \rightarrow \mathbb Z^2$ (a). Edges without labels have identity label $m=(0,0)$. A fragment of its derived graph $G^m$ is shown in (b).  We use graphs with vertex labels as in (a) to depict periodic orbit graphs, and graphs without such vertex labels will record derived graphs, or graphs that are realized in $\mathbb R^d$. \label{fig:gainGraph}}
\end{center}
\end{figure}

%%%%% RIGIDITY ON TORUS
\subsection{Rigidity theory for periodic orbit frameworks on $\Tor^2$}
\subsubsection{Infinitesimal motions of periodic orbit frameworks}
An {\it infinitesimal motion} of a periodic orbit framework $\pofw$ on $\Tor^2$ is an assignment of velocities to each of the vertices, $\bu: V \rightarrow \mathbb{R}^2$, with $u(v_i) = u_i$ such that 
 \begin{equation}
(\bu_i - \bu_j)\cdot(\p_i - \p_j-\bm_eL_0) = 0
\label{eqn:fixTorMot}
\end{equation}
for each edge $e = \{v_i, v_j;\bm_e\} \in E\pog$. An infinitesimal motion preserves the lengths of the bars of the framework. 

A {\it trivial infinitesimal motion} of $\pofw$ on $\Tor^2$ is an infinitesimal motion that preserves the distance between all pairs of vertices, including their copies under periodicity. That is,
\begin{equation}
(\bu_i - \bu_j)\cdot(\p_i - \p_j-\bm_eL_0) = 0
\label{eqn:fixTorTrivMot}
\end{equation}
for all triples  $\{v_i, v_j;\bm_e\}$, $\bm_e \in \mathbb Z^2$.  For any periodic orbit framework $\pofw$ on $\Tor^2$, there will always be a $2$-dimensional space of  trivial infinitesimal motions of the whole framework, namely the space of infinitesimal translations. Rotation is not a trivial motion for periodic orbit frameworks on $\Tor^2$, since we have fixed our representation of the lattice matrix $L_0$ under rotation. If $m_e = 0$ for all edges $e$ in a periodic orbit graph, then infinitesimal rotation is a solution of the system (\ref {eqn:fixTorTrivMot}) since it preserves the distances between $v_1, \dots, v_n$. However, this is a non-trivial motion, since it does not preserve the length of every hypothetical edge $\{v_i, v_j;\bm_e\}$. %\ER{picture of disconnected framework rotating}

If the only infinitesimal motions of a framework $\pofw$ on $\Tor^d$ are trivial (i.e. infinitesimal translations), then it is called {\it infinitesimally rigid}. Otherwise, the framework is called {\it infinitesimally flexible}.

\subsubsection{The fixed torus rigidity matrix}

The {\it fixed torus rigidity matrix} $\R_0\pofw$ is the $|E| \times 2|V|$ matrix that records equations for the space of possible infinitesimal motions of the periodic orbit framework $\pofw$. It has one row for each edge $e = \{v_i, v_j; m_e\}$ of $\pog$ as follows:
\[  \renewcommand{\arraystretch}{1.2}
 \bordermatrix{ &   & v_i &   & v_j &    \cr
 %& & & \vdots & & \cr
& 0 \cdots 0 &\p_i - (\p_j + \bm_eL_0) & 0 \cdots 0 & (\p_j  + \bm_eL_0) - \p_i & 0 \cdots 0 \cr
 %& & & \vdots & & \cr
 },\]
where each entry is actually a $2$-tuple, and the non-zero entries occur in the columns corresponding to vertices $v_i$ and $v_j$. The kernel of this matrix is the space of infinitesimal motions of $\pofw$ on $\Tor^2$.

Since a framework on $\Tor^2$ always has a two-dimensional space of trivial motions (translations), it follows that the kernel of the rigidity matrix always has dimension at least $2$. Furthermore, because a framework is infinitesimally rigid on $\Tor^2$ if and only if the only infinitesimal motions are translations, it follows that a periodic orbit framework $\pofw$ is infinitesimally rigid on the fixed torus $\Tor^2$ if and only if the rigidity matrix $\R_0\pofw$ has rank $2|V|-2$ \cite{ThesisPaper1}.
%\label{thm:fixedMatrixRank}

It follows that a periodic orbit framework with $|E| < 2|V| - 2$ cannot be infinitesimally rigid on $\Tor^2$. 

\begin{ex} Consider the periodic orbit graph $\pog$ shown in Figure \ref{fig:gainGraph}. Let $L_0$ be the matrix generating the torus $\Tor^2$. The rigidity matrix $\R_0\pofw$ will have have six rows, and eight columns (two columns corresponding to the two coordinates of each vertex), as follows:

\[ \bordermatrix{
%row one: names of the vertices
&  v_1  & v_2 & v_3 & v_4 \cr
%row two: names of the edges, matrix entries
\{v_1, v_2; (0, 0)\} & \p_1 - \p_2 & \p_2 - \p_1 &   0 &   0 \cr
\{v_2, v_3; (0, 0)\} &   0 & \p_2 - \p_3 & \p_3 - \p_2 &   0 \cr
\{v_3, v_4; (0, 0)\} &   0 &   0 & \p_3 - \p_4 & \p_4 - \p_3 \cr
\{v_1, v_4; (0, 0)\} & \p_1 - \p_4 &   0 &   0 & \p_4 - \p_1 \cr  
\{v_1, v_3; (-1, 0)\} & \p_1 - \p_3 + (1, 0)L_0 &   0 & \p_3 - \p_1 - (1, 0)L_0&   0 \cr  
\{v_1, v_4; (0, 1)\} & \p_1 - \p_4 - (0,1)L_0 &   0 &   0 & \p_4 - \p_1+(0, 1)L_0   
}. \] 

\qed
\label{ex:rigMatrix}
\end{ex}

A collection of edges $E' \subset E$ of the periodic orbit framework $\pofw$ is called {\it independent} (resp. {\it dependent}) if the corresponding rows of the rigidity matrix are linearly independent (resp. linearly dependent).  For example, any loop edge is dependent on $\Tor^2$, and no more than two copies of an edge of $G$ may be independent. 
We may also refer to a framework $\pofw$ as being independent or dependent. We say a framework with $|E| > 2|V| - 2$ is {\it over-counted}, meaning that it is always dependent.

A periodic orbit framework $\pofw$ that is both infinitesimally rigid and independent on $\Tor^2$  will be called {\it minimally rigid}. %In other words, a minimally  rigid framework on $\Tor^2$ is one that is both infinitesimally rigid and independent. %Such a framework is maximally independent in the sense that adding any new edge will introduce a dependence among the edges, and minimally rigid in the sense that the removal of any edge will result in a framework that is not infinitesimally rigid. %This is analogous to the definition of ``isostatic" for finite frameworks, although we avoid that terminology. 
From the rigidity matrix, we obtain the following necessary condition for minimal rigidity on $\Tor^2$. This is an analogue of Maxwell's original (1864) counting condition for the flexibility of frameworks \cite{Maxwell}.
\begin{thm}
Let $\pofw$ be a minimally rigid periodic orbit framework on $\Tor^2$. Then $G$ is $(2,2)$-tight.
%\begin{enumerate}
%	\item $|E| = 2|V| - 2$, and 
%	\item for all subgraphs $G' \subseteq G$, $|E'| \leq 2|V'| - 2$. 
%\end{enumerate}
\label{thm:perMaxwell}
\end{thm}
%\ER{Change to $(2,2)$-tight?}

The rows of $\R_0\pofw$ corresponding to edges with zero gains are identical to rows in the rigidity matrix of a finite framework, as described in any introduction to rigidity; see \cite{CountingFrameworks} or \cite{SomeMatroids}, for example. Since at most $2|V|-3$ rows can be independent in the finite rigidity matrix, we obtain:
\begin{prop}
Let $\pog$ be a periodic orbit graph with all edges having zero gains,  $\bm = 0$. If $|E| > 2|V| - 3$, then the edges of $\pofw$ are dependent for any realization $\p$. 
\label{prop:zeroGains}
\end{prop} 

\subsubsection{The unit torus and affine transformations}
It was shown independently in \cite{periodicFrameworksAndFlexibility} and \cite{myThesis} that the infinitesimal rigidity of periodic orbit frameworks on the fixed torus is invariant under affine transformations. That is, affine transformations preserve the rank of the fixed torus rigidity matrix $\R_0$. When $L_0$ is the $2 \times 2$ identity matrix, we call $\mathbb R^2 / \mathbb Z^2 L_0$ the {\it unit torus}. For the remainder of this paper we assume that $\T_0^2$ is the unit torus, and we drop the ``$L_0$" from the entries of the rigidity matrix (since $m_eL_0 = m_e$). 

%
%\begin{prop}
%Let $\mathcal F = \pofw$ be a periodic orbit framework on $\Tor^2$, where $L_0$ is the $2 \times 2$ lattice matrix. Let $\mathcal F'$ be the image of $\mathcal F$ under the unique affine transformation of $\mathbb R^2$ which maps $L_0$ to the $2$-dimensional identity matrix $I_{2 \times 2}$. Then $\mathcal F$ is infinitesimally rigid on $\Tor^2$ if and only if $\mathcal F'$ is infinitesimally rigid on the $2$-dimensional unit torus, $\Unit^2$. 
%\label{prop:unitTorus}
%\end{prop}
%

\subsubsection{Generic periodic orbit frameworks}

Since our goal in the remainder of this paper will be to characterize the rigidity  of periodic orbit frameworks based on their periodic orbit graphs, we need a notion of a generic realization on the torus. %Roughly speaking, this will be a realization without any geometric degeneracies. %In addition, since we will be building up larger frameworks from smaller ones, we need a notion of generic which will continue to be generic as we add edges and vertices. 

Let $V$ be a finite set of vertices, and let $\p$ be a realization of these vertices on the $d$-dimensional unit torus $\Tor^2 = [0,1)^2$.  Let $k \in \mathbb Z_+$ be given, and let $K$ be the set of all edges between pairs of vertices of $V$ with gains $\bm_e = (m_{e,1}, m_{e,2})$ where $|m_{e,i}| \leq k$ for $i=1, 2$. Then $K$ is the set of all edges with bounded gains.

Consider a set of edges $E \subset K$ such that, for some realization $\p$, the rows of $\R_0$ corresponding to $E$ are independent. By taking the $\p_i$'s as variables, the determinants of the $|E| \times |E|$ submatrices of these rows will either be identically zero or will define an algebraic variety in $\mathbb R^{2|V|}$. The collection of all such varieties, corresponding to all such subsets $E$ will define a closed set of measure zero, as a finite union of closed sets of measure zero. Let this set be denoted $\mathcal X_k$. The complement of $\mathcal X_k$ in $\mathbb R^{2|V|}$ is an open dense set in $\mathbb R^{2|V|}$, and hence its restriction to the subspace of realizations $\p$ of the vertices $V$ on the unit torus, $[0,1)^{2|V|}$ is also open and dense.

Any realization $\p$ of the vertex set $V$ where $\p \notin \mathcal X_k$ will be called {\it k-generic} (recall that $k$ was the upper bound on the gain assignments).  More generally, we may consider realizations of vertex sets that are $k$-generic for any $k$. By the Baire Category Theorem, the countable intersection 
\[\bigcap_{k \in \mathbb Z} \big( \mathbb R^{2|V|} - \mathcal X_k \big)\]
is dense in $\mathbb R^{2|V|}$, as the intersection of open dense sets in the Baire space $\mathbb R^{2|V|}$ \cite{munkres}. We have shown: 
\begin{prop}
The set of all realizations $\p$ of a vertex set such that the rigidity matrix of any generically rigid periodic orbit graph on these vertices attains its maximal rank is dense in $\mathbb R^{2|V|}$. 
\end{prop}
%\ER{check this!!! ``any periodic orbit graph"?}
We refer to a realization in this set as simply {\it generic}.
%and it is this definition that we use throughout the remainder of this paper. 
All generic frameworks $\pofw$ with the same underlying periodic orbit graph $\pog$ will have the same rigidity properties, a fact captured by the following result, which is analogous to a similar result for finite frameworks, see for example \cite{SomeMatroids}. 

\begin{lem}[\bf Special Position Lemma] 
Let $\pog$ be a periodic orbit graph, and suppose that for some realization $\p_0$ of $\pog$ on $\Tor^2$ the framework $(\pog, \p_0)$ is infinitesimally rigid. Then for all generic realizations $\p$ of $\pog$ on $\Tor^2$, the framework $\pofw$ is infinitesimally rigid.
\label{lem:specialPosition}
\end{lem}
Since the minimal rigidity of a generic periodic orbit framework $\pofw$ does not depend on the specific realization $\p$, we may say that the periodic orbit framework $\pog$ is {\it generically minimally rigid}.

%% T-gain procedure
\subsection{$T$-gain procedure}
The {\it cycle space} of a periodic orbit graph $\pog$ is the cycle space of $G$, denoted $\mathcal C(G)$, which is the vector space generated by the set of all simple cycles of $G$. Elements of  $\mathcal C(G)$ are either simple cycles, or the disjoint union of simple cycles \cite{Diestel}. For any simple cycle $C \in \mathcal C(G)$, we define the {\it net gain} on the cycle $C$ to be the sum of the gains on the edges of the cycle, with sign taken according to the direction of traversal of the edges. We define the {\it gain space} of $\pog$ to be the vector space (over $\mathbb Z$) spanned by the net gains on the cycles of $G$. 

The $T$-gain procedure can be used to easily identify the net gains on the cycles of a periodic orbit graph $\pog$. In particular, the $T$-gain procedure will identify the net gains on a basis for the cycle space of $\pog$, and therefore induces a basis for the gain space of $\pog$. As we will soon see (Section \ref{sec:gainAssignmentsDetermineRigidity}), the rigidity of frameworks on $\Tor^2$ is generically characterized by the net gains on the cycles of the periodic orbit graph. The $T$-gain procedure will thus be an essential tool for the proofs in the rest of the paper. The $T$-gain procedure appears in \cite{TopologicalGraphTheory} for general gain graphs, and it is a specialization of the {\it switching operations} for gain graphs \cite{BiasedGraphsI}. We outline it here for graphs whose gain group is $\mathbb Z^2$. More details can also be found in \cite{ThesisPaper1} or \cite{myThesis}. See Figure \ref{fig:Tvoltage} for a worked example.

\begin{verse}
\begin{figure}[h!]
\begin{center}
\begin{tikzpicture}[->,>=stealth,shorten >=1pt,auto,node distance=2.8cm,thick, font=\footnotesize] 
\tikzstyle{vertex1}=[circle, draw, fill=couch, inner sep=.5pt, minimum width=3.5pt, font=\footnotesize]; 
\tikzstyle{vertex2}=[circle, draw, fill=melon, inner sep=.5pt, minimum width=3.5pt, font=\footnotesize]; 
\tikzstyle{voltage} = [fill=white, inner sep = 0pt,  font=\scriptsize, anchor=center];

	\node[vertex1] (1) at (-1.3,0)  {$1$};
	\node[vertex1] (2) at (1.3,0) {$2$};
	\node[vertex1] (3) at (0, 2) {$3$};
		\draw[thick] (1) -- node[voltage] {$(1,2)$} (2);
		\draw[thick] (2) -- node[voltage] {$(0,1)$} (3);
	\draw[thick] (3) edge  [bend right]  node[voltage] {$(3,1)$} (1);
	\draw[thick] (3) edge  [bend left] node[voltage] {$(1,-1)$} (1);
	
	\node[font=\normalsize] at (0, -1) {(a)};
	
	\pgftransformxshift{4cm};
	
		\node[vertex1] (1) at (-1.3,0)  {$1$};
	\node[vertex1] (2) at (1.3,0) {$2$};
	\node[vertex1] (3) at (0, 2) {$3$};
		\draw[very thick, red] (1) -- node[voltage] {$(1,2)$} (2);
		\draw[thick] (2) -- node[voltage] {$(0,1)$} (3);
	\draw[thick] (3) edge  [bend right]  node[voltage] {$(3,1)$} (1);
	\draw[very thick, red] (3) edge  [bend left] node[voltage] {$(1,-1)$} (1);
	
	\node[blue] at (0, 2.3) {$u$};
	\node[blue] at (-1.6, -.4) {$(1, -1)$};
	\node[blue] at (1.6, -.4) {$(2, 1)$};
	
	\node[font=\normalsize] at (0, -1) {(b)};
	
	\pgftransformxshift{4cm};
	
	\node[vertex1] (1) at (-1.3,0)  {$1$};
	\node[vertex1] (2) at (1.3,0) {$2$};
	\node[vertex1] (3) at (0, 2) {$3$};
		\draw[very thick, red] (1) -- node[voltage] {$(0,0)$} (2);
		\draw[thick] (2) -- node[voltage] {$(2,2)$} (3);
	\draw[thick] (3) edge  [bend right]  node[voltage] {$(4,0)$} (1);
	\draw[very thick, red] (3) edge  [bend left] node[voltage] {$(0,0)$} (1);
	
	\node[font=\normalsize] at (0, -1) {(c)};

\end{tikzpicture}
\caption{A gain graph $\pog$ in (a), with identified tree $T$ (in red), root $u$, and $T$-potentials in (b). The resulting $T$-gain graph $\langle G, \bm_T \rangle$ is shown in (c). The gain space is now seen to be generated by the elements $(4,0)$ and $(2,2)$, hence the gain space is $2\mathbb Z \times 2\mathbb Z$. \label{fig:Tvoltage}}
\end{center}
\end{figure}
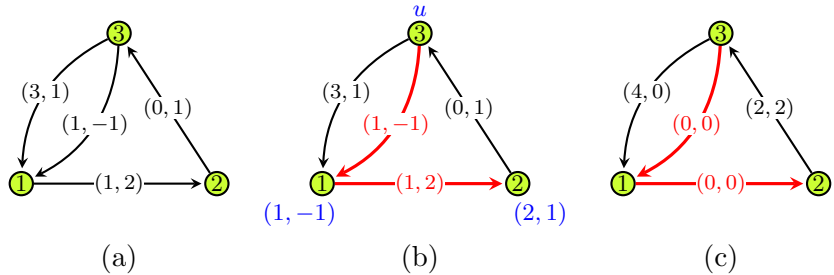\end{verse}

\noindent {\bf $T$-gain Procedure}
\begin{enumerate}
	\item Let $\pog$ be a gain graph, where $G$ is a connected graph. Select an arbitrary spanning tree $T$ of $G$, and choose a vertex $u$ to be the root vertex. 
	\item For every vertex $v$ in $G$, there is a unique path in the tree $T$ from the root $u$ to $v$. Denote the net gain along that path by $\bm(v, T)$, and we call this the {\it $T$-potential} of $v$. \index{T-gain procedure!T-potential} Compute the $T$-potential of every vertex $v$ of $G$. 
	\item Let $e$ be a forward-directed edge of $G$ with initial vertex $v$ and terminal vertex $w$. Define the {\it $T$-gain} of $e$, $\bm_T(e)$ to be 
	$$\bm_T(e) = \bm(v, T) +  \bm(e) - \bm(w, T).$$ 
Compute the $T$-gain of every edge in $G$. Note that the $T$-gain of every edge of the spanning tree will be zero. 
%	\item Contract the graph along the spanning tree to obtain $|E| - (|V|-1)$ loops at the root vertex $u$ (there are $|V|-1$ edges as part of the spanning tree). The gains on these loops will generate the gain space. In other words, the gains on all of the edges of the graph that are not contained in $T$ will generate the gain space. 
\end{enumerate}

\begin{thm}[\cite{TopologicalGraphTheory}] 
Let $\pog$ be a periodic orbit graph, and let $\tgT$ be the same periodic orbit graph after the $T$-gain procedure. Then $\pog$ and $\tgT$ have the same gain space. 
\end{thm}

It is also true that the corresponding derived graphs are isomorphic. 
\begin{thm}[\cite{TopologicalGraphTheory}]
Let $\pog$ be a gain graph, let $u$ be any vertex of $G$, and let $T$ be any spanning tree of $G$. Then the derived graph $G^{\bm_T}$ corresponding to $\langle G, \bm_T \rangle$ is isomorphic to the derived graph $G^{\bm}$. 
\label{thm:TGainIsomorphic}
\end{thm}

We say that the graphs $\pog$ and $\langle G, \bm_T \rangle$ are {\it $T$-gain related} and we write $\pog \sim \langle G, \bm_T \rangle$. More broadly, we say that $\pog$ and $ \langle G, \bm' \rangle$ are {\it $T$-gain equivalent} if $\pog \sim \langle G, \bm_T \rangle$ and $ \langle G, \bm' \rangle \sim \langle G, \bm_T \rangle$ for some choice of spanning tree $T$. In fact, if this is true for one spanning tree, it must be true for all choices of spanning tree, since the $T$-gain procedure preserves the net gains on all cycles. $T$-gain equivalence can easily be shown to be an equivalence relation on the set of all gain assignments on a graph $G$ \cite{ThesisPaper1}.

Most important to the remainder of the paper is the following result, which establishes that $T$-gain equivalent periodic orbit graphs have the same generic rigidity properties. 
\begin{thm}
The periodic orbit graph $\pog$ is generically rigid on $\Tor^2$ if and only if $\tgT$ is generically rigid on $\Tor^2$. 
\label{thm:TgainsPreserveRigidity}
\end{thm}

Theorem \ref{thm:TgainsPreserveRigidity} is proved in \cite{ThesisPaper1}, so we only sketch the key ideas here.  It is shown that, for a generic position $p$, $\rank R_0(\pog, p) = \rank R_0(\langle G, m_T \rangle, p')$, where $p': V \rightarrow \mathbb R^2$ is defined by $p_i' = p_i + m_T(v_i)$.  The basic idea is that the $T$-gain procedure changes the representatives of the vertices used in the rigidity matrix, which, together with the new gains, leaves the rigidity matrix (and its rank) unchanged. We then conclude the generic result using the Special Position Lemma (Lemma \ref{lem:specialPosition}).

\section{Generating minimally rigid frameworks on $\Tor^2$}
\label{sec:generatingIsostaticFrameworks}
We now describe inductive methods for generating infinitesimally rigid frameworks on the fixed torus $\Tor^2$.

%%%%%%%%%%%%%%%%%
\subsection{Inductive constructions}
\label{sec:inductiveConstructions}

Let $\pofw$ be an infinitesimally rigid periodic orbit framework on $\Tor^2$. It is possible to construct other infinitesimally rigid frameworks from $\pofw$ using {\it gain-preserving vertex additions} and {\it edge splits}. We present here the details of these inductive constructions, which we will also call {\it gain-preserving Henneberg moves} after their finite counterparts which were developed by Henneberg \cite{Henneberg}. The proofs of Propositions \ref{prop:vertexAddition}, \ref{prop:edgeSplit} and \ref{prop:reverseEdgeSplit} follow the basic method of their finite counterparts shown in \cite{SomeMatroids}. Analogous statements for the partially flexible torus appear in \cite{NixonRoss} using the same proof methods.

%%%% VERTEX ADDITION %%%%
\subsubsection{Gain-preserving vertex addition}
\begin{verse}
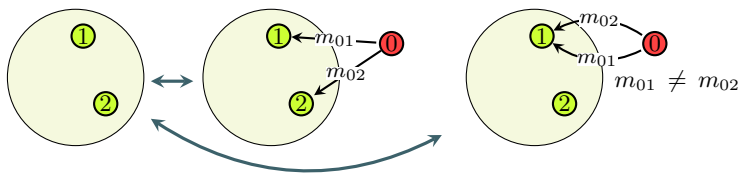
\begin{figure}[h!]
\begin{center}
\begin{tikzpicture}[->,>=stealth,shorten >=1pt,auto,node distance=2.8cm,thick, font=\footnotesize] 
\tikzstyle{vertex1}=[circle, draw, fill=couch, inner sep=.5pt, minimum width=3.5pt, font=\footnotesize]; 
\tikzstyle{vertex2}=[circle, draw, fill=melon, inner sep=.5pt, minimum width=3.5pt, font=\footnotesize]; 
\tikzstyle{voltage} = [fill=white, inner sep = 0pt,  font=\scriptsize, anchor=center];

	\draw (0.1,0.05) circle (.9cm); 
	\fill[cloud] (0.1,0.05) circle (.9cm);
		
	\node[vertex1] (1) at (.2,.6)  {$1$};
	\node[vertex1] (2) at (.5,-.3) {$2$};
	\path[<->, bluey, very thick] (1.1, -0.5) edge [bend right] (5, -.7);
%with added vertex
	\pgftransformxshift{1.3cm}
	\draw[<->, very thick, bluey] (-.2, 0) -- (.4, 0);
	\pgftransformxshift{1.3cm}

	\draw (0.1,0.05) circle (.9cm); 
	\fill[cloud] (0.1,0.05) circle (.9cm);
		
	\node[vertex1] (1) at (.2,.6)  {$1$};
	\node[vertex1] (2) at (.5,-.3) {$2$};
	\node[vertex2] (3) at (1.7, .5) {$0$};

	\draw[thick, <-] (1) --  node[voltage] {$\bm_{01}$} (3);
		\draw[thick] (3) -- node[voltage] {$\bm_{02}$} (2);
					
	\pgftransformxshift{3.5cm}
	\draw (0.1,0.05) circle (.9cm); 
	\fill[cloud] (0.1,0.05) circle (.9cm);
		
	\node[vertex1] (1) at (.2,.6)  {$1$};
	\node[vertex1] (2) at (.5,-.3) {$2$};
		\node[vertex2] (3) at (1.7, .5) {$0$};

     \draw[thick, <-] (1) edge [bend right] node[voltage] {$\bm_{01}$} (3);
     \draw[thick,<-] (1) edge [bend left] node[voltage] {$\bm_{02}$} (3);

	\node[text width=2.5cm, text centered] at (2,0) {$\bm_{01} \neq \bm_{02}$};

\end{tikzpicture} 
\caption{Gain-preserving vertex addition. The large circular region represents a generically rigid periodic orbit graph.  \label{fig:vertexAddition}}
\end{center}
\end{figure}\end{verse}

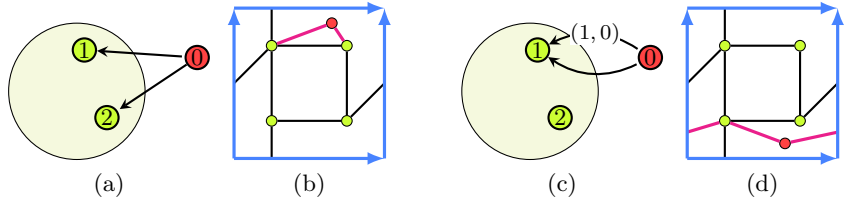
\begin{figure}\begin{center}
%VERTEX ADDITIONS
%two neighbours
\subfloat[]{\begin{tikzpicture}[->,>=stealth,shorten >=1pt,auto,node distance=2.8cm,thick, font=\footnotesize] 
\tikzstyle{vertex1}=[circle, draw, fill=couch, inner sep=.5pt, minimum width=3.5pt, font=\footnotesize]; 
\tikzstyle{vertex2}=[circle, draw, fill=melon, inner sep=.5pt, minimum width=3.5pt, font=\footnotesize]; 
\tikzstyle{gain} = [fill=white, inner sep = 0pt,  font=\scriptsize, anchor=center];

	\draw (0.1,0.05) circle (.9cm); 
	\fill[cloud] (0.1,0.05) circle (.9cm);
		
	\node[vertex1] (1) at (.2,.6)  {$1$};
	\node[vertex1] (2) at (.5,-.3) {$2$};
	\node[vertex2] (3) at (1.7, .5) {$0$};

	\draw[thick, <-] (1) --   (3);
		\draw[thick] (3) --  (2);
\end{tikzpicture}}\hspace{.3cm}%
\subfloat[]{\begin{tikzpicture}[auto, node distance=2cm]
\tikzstyle{vertex1}=[circle, draw, fill=couch, inner sep=1pt, minimum width=3pt]; 
\tikzstyle{vertex2}=[circle, draw, fill=lips, inner sep=1pt, minimum width=3pt]; 
\tikzstyle{vertex3}=[circle, draw, fill=melon, inner sep=1pt, minimum width=3pt]; 
\tikzstyle{vertex4}=[circle, draw, fill=bluey, inner sep=1pt, minimum width=3pt]; 

\node[vertex1, minimum width=4pt] (1) at (0,0) {};
\node[vertex1, minimum width=4pt] (2) at (1,0) {};
\node[vertex1, minimum width=4pt] (3) at (1,-1) {};
\node[vertex1, minimum width=4pt] (4) at (0,-1) {};

\draw[thick] (1) -- (2) -- (3) -- (4) -- (1) ;
\draw[thick] (1) -- (0, .5);
\draw[thick] (4) -- (0, -1.5);
\draw[thick] (3) -- (1.5, -.5);
\draw[thick] (1) -- (-.5, -.5);

\node[vertex3, minimum width=4pt] (5) at (.8, .3) {};
\draw[very thick, pink] (5) -- (1) (5) -- (2);

\pgfsetarrowsend{latex} 
	\draw[sky, very thick] (-.5, -1.5) -- (1.5, -1.5); 
	\draw[sky, very thick] (-.5, .5) -- (1.5, .5);
	\draw[sky, very thick] (-.5, -1.5) -- (-.5, .5);
	\draw[sky, very thick] (1.5, -1.5) -- (1.5, .5);
\pgfsetarrowsend{} 
\end{tikzpicture}}\hspace{1cm}%
%one neighbour
\subfloat[]{\begin{tikzpicture}[->,>=stealth,shorten >=1pt,auto,node distance=2.8cm,thick, font=\footnotesize] 
\tikzstyle{vertex1}=[circle, draw, fill=couch, inner sep=.5pt, minimum width=3.5pt, font=\footnotesize]; 
\tikzstyle{vertex2}=[circle, draw, fill=melon, inner sep=.5pt, minimum width=3.5pt, font=\footnotesize]; 
\tikzstyle{gain} = [fill=white, inner sep = 0pt,  font=\scriptsize, anchor=center];
	\draw (0.1,0.05) circle (.9cm); 
	\fill[cloud] (0.1,0.05) circle (.9cm);
		
	\node[vertex1] (1) at (.2,.6)  {$1$};
	\node[vertex1] (2) at (.5,-.3) {$2$};
		\node[vertex2] (3) at (1.7, .5) {$0$};

     \draw[thick, <-] (1) edge [bend right]  (3);
     \draw[thick,<-] (1) edge [bend left] node[gain] {$(1,0)$} (3);

	%\node[text width=2.5cm, text centered] at (2,0) {$m_1 \neq m_{1'}$};
\end{tikzpicture}}\hspace{.3cm}%
\subfloat[]{\begin{tikzpicture}[auto, node distance=2cm]
\tikzstyle{vertex1}=[circle, draw, fill=couch, inner sep=1pt, minimum width=3pt]; 
\tikzstyle{vertex2}=[circle, draw, fill=lips, inner sep=1pt, minimum width=3pt]; 
\tikzstyle{vertex3}=[circle, draw, fill=melon, inner sep=1pt, minimum width=3pt]; 
\tikzstyle{vertex4}=[circle, draw, fill=bluey, inner sep=1pt, minimum width=3pt]; 

\node[vertex1, minimum width=4pt] (1) at (0,0) {};
\node[vertex1, minimum width=4pt] (2) at (1,0) {};
\node[vertex1, minimum width=4pt] (3) at (1,-1) {};
\node[vertex1, minimum width=4pt] (4) at (0,-1) {};

\draw[thick] (1) -- (2) -- (3) -- (4) -- (1) ;
\draw[thick] (1) -- (0, .5);
\draw[thick] (4) -- (0, -1.5);
\draw[thick] (3) -- (1.5, -.5);
\draw[thick] (1) -- (-.5, -.5);

\node[vertex3, minimum width=4pt] (5) at (.8, -1.3) {};
\draw[very thick, pink] (5) -- (4) (5) -- (1.5, -1.15) (-.5, -1.15) -- (4);% (5) -- (2);

\pgfsetarrowsend{latex} 
	\draw[sky, very thick] (-.5, -1.5) -- (1.5, -1.5); 
	\draw[sky, very thick] (-.5, .5) -- (1.5, .5);
	\draw[sky, very thick] (-.5, -1.5) -- (-.5, .5);
	\draw[sky, very thick] (1.5, -1.5) -- (1.5, .5);
\pgfsetarrowsend{} 
\end{tikzpicture}}
\caption{Gain-preserving vertex addition on the torus. Parts (a) and (b) denote examples of vertex addition in the case that $v_0$ is adjacent to two distinct vertices. In (c) and (d), $v_0$ is adjacent to only one vertex in $\pog$. \label{fig:vertexAdditionTorus}}
\end{center}
\end{figure}

Given a periodic orbit graph $\pog$, a {\it (gain-preserving) vertex addition} is the addition of a single new vertex $v_0$ to $V=V(G) = V\pog$, and the edges $\{v_0, v_{i_1}; \bm_{01}\}$ and $\{v_0, v_{i_2}; \bm_{02}\}$ to $E\pog$, such that $\bm_{01} \neq \bm_{02}$ whenever $v_{i_1} = v_{i_2}$ (see Figure \ref{fig:vertexAddition}). Provided that $v_{i_1} \neq v_{i_2}$, by definition, $\bm_{01}$ and $\bm_{02}$ may always taken to be $(0,0)$, since this is simply the usual finite vertex addition. Examples are shown in Figure \ref{fig:vertexAdditionTorus}. 

\begin{prop}[Gain-preserving vertex addition]
Let $\pog$ be a periodic orbit graph with $n$ vertices, and let $\langle G', \bm' \rangle$ be the graph created by performing a vertex addition on $\pog$, adding the vertex $v_0$ to $G$. Let $\p$ be a set of $n$ generic points, and let $\p'$ be a set of $n+1$ generic points, agreeing with $\p$ on the first $n$. Then $\pofw$ is independent if and only if $\pofwpr$ is independent.
%\ER{Original statement: Then the rows of $\R_0\pofw$ are independent if and only if the rows of $\R_0\pofwpr$ are independent.}
%For a generic choice of $\p:V\rightarrow \Tor^2$, and with $\p_0$ chosen generically with respect to $\p$, the rows of $\R_0\pofw$ are independent if and only if the rows of $\R_0\pofwpr$ are independent, where $\p' = \p \cup \p_0$. 
\label{prop:vertexAddition}
\end{prop}

\begin{proof}
 Suppose that the vertex addition adds $v_0$ and the edges $\{v_0, v_{i_1}; \bm_{01}\}$ and $\{v_0, v_{i_2}; \bm_{02}\}$, where $v_{i_1}$ and $v_{i_2} \in V$ may or may not be the same vertex. The rigidity matrix of $\pogp$ is %
\[\R_0\pofwpr = \bordermatrix{
%row one: names of the vertices
 &  v_0 & | & v_1 & \cdots & v_{|V|} \cr
%row two: names of the edges, matrix entries
\hfill e_1 \ & 0  & | & &  & \cr
\hfill \vdots \ & \vdots & | & & \R_0\pofw &  \cr
\hfill e_{|E|} \ & 0 & | & & &  \cr
\hline \cr
\{v_0, v_{i_1}; \bm_{01}\} & \p_0 - \p_{i_1} - \bm_{01} &| & & \cdots & \cr
\{v_0, v_{i_2}; \bm_{02}\} & \p_0 - \p_{i_2} - \bm_{02} &|  & & \cdots & }.
%\grey{(\p_1, \p_4; 0, 1)} & \p_1 - \p_4 - (0,1) &\bf  0 &\bf  0 & \p_4 - \p_1+(0, 1) \\  
 \]
%
%\[R_0\pofwpr = \begin{array}{|c|c|ccc|} \hline
%%row one: names of the vertices
% &  \grey{v_0}  & \grey{v_1} & \grey{\cdots} & \grey{v_{|V|}} \\ \hline
%%row two: names of the edges, matrix entries
%\grey{e_1} & 0 &  &  & \\ 
%\grey{\vdots} & \vdots &  & R_0\pofw &  \\ 
%\grey{e_{|E|}} & 0 &  & & \\ \hline
%\grey{(v_0, v_{i_1}; \bm_{01})} & \p_0 - \p_{i_1} - \bm_{01} & \cdots & & \\  
%\grey{(v_0, v_{i_2}; \bm_{02})} & \p_0 - \p_{i_2} - \bm_{02} & \cdots & & \\  
%%\grey{(\p_1, \p_4; 0, 1)} & \p_1 - \p_4 - (0,1) &\bf  0 &\bf  0 & \p_4 - \p_1+(0, 1) \\  
%\hline
%\end{array}  \]
%
Toward a contradiction, suppose that the rows of $\R_0\pofwpr$ are dependent. Then the columns of $\R_0\pofwpr$ corresponding to $v_0$ provide the relationship:
\[\omega_{01}(\p_0 - \p_{i_1} - \bm_{01}) + \omega_{02}(\p_0 - \p_{i_2} - \bm_{02}) = 0 \]
for some $\omega_{01}, \omega_{02} \in \mathbb R$. The vectors $(\p_0 - \p_{i_1} - \bm_{01})$ and $(\p_0 - \p_{i_2} - \bm_{02})$ are linearly independent  if and only if the points $p_0$, $p_{i_1} + \bm_{01}$ and $p_{i_2} + \bm_{02}$ are {\it not} collinear. It follows from the definition of generic realizations that if $p'$ is generic, then no three points of the form $p_1+m_1$, $p_2+m_2$, $p_3+m_3$ are collinear (the three rows of the rigidity matrix corresponding to the edges of the triangle formed by connecting these points are independent). 

%If $v_{i_1} \neq v_{i_2}$, or if $v_{i_1} = v_{i_2}$ but $\bm_{01} \neq \bm_{02}$, then these points are not collinear, since we chose $\p_0$ generically with respect to $\p$. 
Hence $\omega_{01} = \omega_{02} = 0$, which creates a dependence among the rows of $\R_0\pofw$, contradicting our assumption that the rows of $\R_0\pofw$ were independent. The argument reverses for the converse. (Assume the rows of $\R\pofwpr$ are dependent and proceed from there.) %
\end{proof}

Note that Proposition \ref{prop:vertexAddition} also has a geometric meaning. In fact, the proof of that result was geometric in nature, in the sense that we chose $p$ so that the points $p_0$, $p_{i_1} + \bm_{01}$ and $p_{i_2} + \bm_{02}$ were not collinear in $\mathbb R^2$ (in fact we chose $p$ to be generic, and the non-collinearity followed).

%%%% EDGE SPLITTING %%%%
\subsubsection{Gain-preserving edge splitting}
Let $\pog$ be a periodic orbit graph, and let $e = \{v_{i_1}, v_{i_2}; \bm_e\}$ be an edge of $\pog$. A {\it gain-preserving edge split} $\pogp$ of $\pog$ is a graph with vertex set $V \cup \{v_0\}$ and edge set consisting of all of the edges of $E\pog$ except $e$, together with the edges \[ \{v_0, v_{i_1}; (0,0)\}, \{v_0, v_{i_2}; \bm_e\}, \{v_0, v_{i_3}; \bm_{03}\}\]
where $v_{i_1} \neq v_{i_2}, v_{i_3}$, and $\bm_{03} \neq \bm_{e}$ if $v_{i_2}=v_{i_3}$ (see Figure \ref{fig:edgeSplit}).

\begin{verse}
\begin{figure}[h!]
\begin{center}
\begin{tikzpicture}[->,>=stealth,shorten >=1pt,auto,node distance=2.8cm,thick, font=\footnotesize] 
\tikzstyle{vertex1}=[circle, draw, fill=couch, inner sep=.5pt, minimum width=3.5pt, font=\footnotesize]; 
\tikzstyle{vertex2}=[circle, draw, fill=melon, inner sep=.5pt, minimum width=3.5pt, font=\footnotesize]; 
\tikzstyle{voltage} = [fill=white, inner sep = 0pt,  font=\scriptsize, anchor=center];

	\draw (0.1,0.05) circle (.9cm); 
	\fill[cloud] (0.1,0.05) circle (.9cm);
		
	\node[vertex1] (1) at (.2,.7)  {$1$};
	\node[vertex1] (2) at (.6,-.4) {$2$};
	   \node[vertex1] (4) at (-.4, -.1) {$3$};
	
	\draw[thick] (1) -- node[voltage, fill=cloud] {$\bm_e$} (2);
	\path[bluey, very thick] (1.1, -0.5) edge [bend right] (5.6, -.7);
	\pgftransformxshift{1.5cm}
	\draw[very thick, bluey] (-.2, 0) -- (.4, 0);
	\pgftransformxshift{1.5cm}

	\draw (0.1,0.05) circle (.9cm); 
	\fill[cloud] (0.1,0.05) circle (.9cm);
		
	\node[vertex1] (1) at (.2,.7)  {$1$};
	\node[vertex1] (2) at (.6,-.4) {$2$};
	   \node[vertex1] (4) at (-.4, -.1) {$3$};
	\node[vertex2] (3) at (1.7, .5) {$0$};

	\draw[thick] (1) --  node[voltage] {$(0,0)$} (3);
		\draw[thick] (3) -- node[voltage] {$\bm_e$} (2);
		\draw[thick] (3) -- node[voltage] {$\bm_{03}$} (4);
		
	%	\node at (2.2,-0.5) {$m_1 + m_2 = m_e$};
	\pgftransformxshift{3.5cm}
	\draw (0.1,0.05) circle (.9cm); 
	\fill[cloud] (0.1,0.05) circle (.9cm);
		
	\node[vertex1] (1) at (.2,.7)  {$1$};
	\node[vertex1] (2) at (.6,-.4) {$2$};
	   \node[vertex1] (4) at (-.4, -.1) {$3$};
	\node[vertex2] (3) at (1.7, .5) {$0$};

	\draw[thick] (1) -- node[voltage] {$(0,0)$} (3);
	\draw[thick] (3) edge [bend right]  node[voltage] {$\bm_{03}$} (2);
	\draw[thick] (3) edge [bend left] node[voltage] {$\bm_{e}$} (2);

	\node[text width=2.5cm, text centered] at (2.5,-0.75) { $\bm_{03} \neq \bm_{e}$};
\end{tikzpicture}
\caption{Gain-preserving edge split. \label{fig:edgeSplit}}
\end{center}
\end{figure}
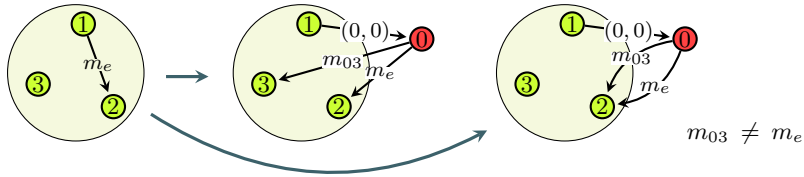\end{verse}

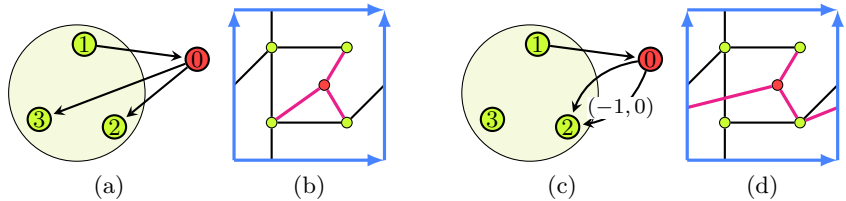
\begin{figure}\begin{center}
\subfloat[]{\begin{tikzpicture}[->,>=stealth,shorten >=1pt,auto,node distance=2.8cm,thick, font=\footnotesize] 
\tikzstyle{vertex1}=[circle, draw, fill=couch, inner sep=.5pt, minimum width=3.5pt, font=\footnotesize]; 
\tikzstyle{vertex2}=[circle, draw, fill=melon, inner sep=.5pt, minimum width=3.5pt, font=\footnotesize]; 
\tikzstyle{gain} = [fill=white, inner sep = 0pt,  font=\scriptsize, anchor=center];
	\draw (0.1,0.05) circle (.9cm); 
	\fill[cloud] (0.1,0.05) circle (.9cm);
		
	\node[vertex1] (1) at (.2,.7)  {$1$};
	\node[vertex1] (2) at (.6,-.4) {$2$};
	   \node[vertex1] (4) at (-.4, -.3) {$3$};
	\node[vertex2] (3) at (1.7, .5) {$0$};

	\draw[thick] (1) --   (3);
		\draw[thick] (3) --  (2);
		\draw[thick] (3) -- (4);
		
		%\node at (2.2,-0.5) {$m_1 + m_2 = m_e$};
\end{tikzpicture}}\hspace{.3cm}%
\subfloat[]{\begin{tikzpicture}[auto, node distance=2cm]
\tikzstyle{vertex1}=[circle, draw, fill=couch, inner sep=1pt, minimum width=3pt]; 
\tikzstyle{vertex2}=[circle, draw, fill=lips, inner sep=1pt, minimum width=3pt]; 
\tikzstyle{vertex3}=[circle, draw, fill=melon, inner sep=1pt, minimum width=3pt]; 
\tikzstyle{vertex4}=[circle, draw, fill=bluey, inner sep=1pt, minimum width=3pt]; 

\node[vertex1, minimum width=4pt] (1) at (0,0) {};
\node[vertex1, minimum width=4pt] (2) at (1,0) {};
\node[vertex1, minimum width=4pt] (3) at (1,-1) {};
\node[vertex1, minimum width=4pt] (4) at (0,-1) {};

\draw[thick] (1) -- (2)  (3) -- (4) -- (1) ;
\draw[thick] (1) -- (0, .5);
\draw[thick] (4) -- (0, -1.5);
\draw[thick] (3) -- (1.5, -.5);
\draw[thick] (1) -- (-.5, -.5);

\node[vertex3, minimum width=4pt] (5) at (.7, -.5) {};
\draw[very thick, pink] (5) -- (4) (5) -- (2) (5) -- (3);

\pgfsetarrowsend{latex} 
	\draw[sky, very thick] (-.5, -1.5) -- (1.5, -1.5); 
	\draw[sky, very thick] (-.5, .5) -- (1.5, .5);
	\draw[sky, very thick] (-.5, -1.5) -- (-.5, .5);
	\draw[sky, very thick] (1.5, -1.5) -- (1.5, .5);
\pgfsetarrowsend{} 
\end{tikzpicture}}\hspace{1cm}%
% two neighbours
\subfloat[]{\begin{tikzpicture}[->,>=stealth,shorten >=1pt,auto,node distance=2.8cm,thick, font=\footnotesize] 
\tikzstyle{vertex1}=[circle, draw, fill=couch, inner sep=.5pt, minimum width=3.5pt, font=\footnotesize]; 
\tikzstyle{vertex2}=[circle, draw, fill=melon, inner sep=.5pt, minimum width=3.5pt, font=\footnotesize]; 
\tikzstyle{gain} = [fill=white, inner sep = 0pt,  font=\scriptsize, anchor=center];
	\draw (0.1,0.05) circle (.9cm); 
	\fill[cloud] (0.1,0.05) circle (.9cm);
		
	\node[vertex1] (1) at (.2,.7)  {$1$};
	\node[vertex1] (2) at (.6,-.4) {$2$};
	   \node[vertex1] (4) at (-.4, -.3) {$3$};
	\node[vertex2] (3) at (1.7, .5) {$0$};

	\draw[thick] (1) --  (3);
	\draw[thick] (3) edge [bend right]  (2);
	\draw[thick] (3) edge [bend left] node[gain] {$(-1,0)$} (2);

%	\node[text width=2.5cm, text centered] at (2.5,-0.75) {$m_1 + m_2 = m_e$ $m_2 \neq m_{2'}$};
\end{tikzpicture}}\hspace{.3cm}%
\subfloat[]{\begin{tikzpicture}[auto, node distance=2cm]
\tikzstyle{vertex1}=[circle, draw, fill=couch, inner sep=1pt, minimum width=3pt]; 
\tikzstyle{vertex2}=[circle, draw, fill=lips, inner sep=1pt, minimum width=3pt]; 
\tikzstyle{vertex3}=[circle, draw, fill=melon, inner sep=1pt, minimum width=3pt]; 
\tikzstyle{vertex4}=[circle, draw, fill=bluey, inner sep=1pt, minimum width=3pt]; 

\node[vertex1, minimum width=4pt] (1) at (0,0) {};
\node[vertex1, minimum width=4pt] (2) at (1,0) {};
\node[vertex1, minimum width=4pt] (3) at (1,-1) {};
\node[vertex1, minimum width=4pt] (4) at (0,-1) {};

\draw[thick] (1) -- (2)  (3) -- (4) -- (1) ;
\draw[thick] (1) -- (0, .5);
\draw[thick] (4) -- (0, -1.5);
\draw[thick] (3) -- (1.5, -.5);
\draw[thick] (1) -- (-.5, -.5);

\node[vertex3, minimum width=4pt] (5) at (.7, -.5) {};
\draw[very thick, pink]  (5) -- (2) (5) -- (3) (5) -- (-.5, -.8) (1.5, -.8) -- (3);

\pgfsetarrowsend{latex} 
	\draw[sky, very thick] (-.5, -1.5) -- (1.5, -1.5); 
	\draw[sky, very thick] (-.5, .5) -- (1.5, .5);
	\draw[sky, very thick] (-.5, -1.5) -- (-.5, .5);
	\draw[sky, very thick] (1.5, -1.5) -- (1.5, .5);
\pgfsetarrowsend{} 
\end{tikzpicture}}
\caption{Gain-preserving edge splits on the torus. Parts (a) and (b) depict the case where $v_0$ is adjacent to three distinct vertices in $\pog$, while (c) and (d) illustrate the case of only two distinct neighbours.  \label{fig:edgeSplitTorus}}
\end{center}
\end{figure}

Gain-preserving edge splits, and reverse gain-preserving edge splits,  preserve infinitesimal rigidity. We will show this in two parts, by first showing that the gain-preserving edge split preserves independence of the rows of the rigidity matrix. 

\begin{prop}[Gain-preserving Edge Split]
Let $\pog$ be a periodic orbit graph with $n$ vertices, and let $\pogp$ be an edge split of it. Let $\p$ be a set of $n$ generic points, and let $\p'$ be a set of $n+1$ generic points, agreeing with $\p$ on the first $n$. If $\pofw$ is independent, then $\pofwpr$ is independent too.
%\ER{Original phrasing: Let $\p:V\rightarrow \Tor^2$ be a generic realization of $V(G)$ on $\Tor^2$, and let $\p_0$ be chosen generically with respect to $\p$.  If the rows of $\R_0\pofw$ are independent, then the rows of $\R_0\pofwpr$ are also independent, where $\p' = \p \cup \p_0$. }
\label{prop:edgeSplit}
\end{prop}

\begin{proof}

Suppose that $\p$ is a generic realization of the vertices of $G$ on $\Tor^2$, with no vertex on the boundary of $\Tor^2$, and place $\p_0$ on the edge connecting the points $p_{i_1}$ and $p_{i_2}+m_e$, where the segment containing $p_{i1}$ and $p_0$ lies in $[0,1)^2$.  Without loss of generality, suppose that $e_1 = \{v_{i_1}, v_{i_2}; \bm_e\}$ is the split edge.  Let $\R_0\pofw - e_1$ denote the rigidity matrix of $\pofw$ without the row corresponding to the edge $e_1$. The rigidity matrix $\R_0\pofwpr$ becomes:
\[  \bordermatrix{
%row one: names of the vertices
 &  v_0  & & v_1 & v_2 &\cdots & v_{|V|} \cr
%row two: names of the edges, matrix entries
\hfill e_2 \  & 0 &&  & &  & \cr 
\hfill \vdots \  & \vdots &&  & \R_0\pofw - e_1 & &  \cr 
\hfill e_{|E|} \ & 0 &&  & & & \cr \hline \cr
\hfill \{v_0, v_{i_1}; 0\} & \p_0 - \p_{i_1} && \p_{i_1} - \p_0 &  0 & \cdots & \cr 
\hfill \{v_0, v_{i_2}; \bm_e\} & \p_0 - \p_{i_2} - \bm_e && 0  & \p_{i_2} - \p_0 + \bm_e & \cdots  & \cr
\{v_0, v_{i_3}; \bm_{03}\} & \p_0 - \p_{i_3} - \bm_{03} & &0  &  &  \cdots & 
}. \]
%
%\[ \begin{array}{|c|ccccc|} \hline
%%row one: names of the vertices
%R_0\pofwpr &  \grey{v_0}  & \grey{v_1} & \grey{v_2} & \grey{\cdots} & \grey{v_{|V|}} \\ \hline
%%row two: names of the edges, matrix entries
%\grey{e_1} & 0 &  & &  & \\ 
%\grey{\vdots} & \vdots &  & R_0\pofw & &  \\ 
%\grey{e_{|E|}} & 0 &  & & & \\ \hline
%\grey{(v_0, v_{i_1}; 0)} & \p_0 - \p_{i_1} & \p_{i_1} - \p_0  & 0 & \cdots & \\  
%\grey{(v_0, v_{i_2}; \bm_e)} & \p_0 - \p_{i_2} - \bm_e & 0  & \p_{i_2} - \p_0 + \bm_e & \cdots  & \\
%\grey{(v_0, v_{i_3}; \bm_{03})} & \p_0 - \p_{i_3} - \bm_{03} & 0  & 0 &  \cdots & \\  
%\hline
%\end{array}  \]
%
Toward a contradiction, suppose that there is a non-trivial dependence among the rows of $\R_0\pofwpr$. That is, suppose that there exists $\omega \neq 0$ where
$\omega = [\begin{array}{cccccc}  \omega_2 & \cdots & \omega_{|E|} & \omega_{01} & \omega_{02} & \omega_{03} \end{array}]$ such that
\[\omega \cdot \R_0\pofwpr = 0.\]
%for $\omega \neq 0$ where
%$\omega = [\begin{array}{cccccc}  \omega_2 & \cdots & \omega_{|E|} & \omega_{01} & \omega_{02} & \omega_{03} \end{array}].$

The vector equation describing the first two columns of this expression (the columns corresponding to $v_0$) becomes:
\[\omega_{01}(\p_0 - \p_{i_1} ) + \omega_{02}(\p_0 - \p_{i_2} - \bm_e) + \omega_{03}(\p_0 - \p_{i_3} - \bm_{03}) = 0.\]
Not all of $\omega_{01}, \omega_{02}, \omega_{03}$ can be $0$, otherwise we would immediately have a nontrivial dependence among the rows of $\R_0\pofw$, contradicting our hypothesis.  

As a consequence of placing $\p_0$ along the edge connecting $p_{i_1}$ and $p_{i_2} +m_e$, the vectors $(\p_0 - \p_{i_1})$ and $(\p_0 - \p_{i_2} - \bm_e)$ are parallel.  However, $(\p_0 - \p_{i_3} - \bm_{03})$ is in a distinct direction, and therefore $\omega_{03} = 0$. % 
Since both of these vectors are again parallel to the deleted edge, we have
\[\omega_{01}(\p_0 - \p_{i_1} ) =  -\omega_{02}(\p_0 - \p_{i_2} - \bm_e) = \omega_{12}(\p_{i_1} - \p_{i_2} - \bm_e)\]
for some scalar $\omega_{12} \neq 0$. But then the coefficients of the rows of $\R_0\pofwpr$ corresponding to the edges in $E \cap E'$, together with $\omega_{12}$, provide a dependence among the rows of $\R_0\pofw$, which contradicts our hypothesis.

By the Special Position Lemma (Lemma \ref{lem:specialPosition}),  we conclude that the edges of $(\pogp, \p')$ are generically independent, since the edges are independent for a special position of $\p_0$. %
\end{proof}

The {\it reverse gain-preserving edge split}  will delete a 3-valent vertex, and add an edge between two of the vertices formerly adjacent to that vertex (Figure \ref{fig:reverseEdgeSplit}). In particular, if $v_0$ is the 3-valent vertex incident to the edges \[\{v_0, v_{i_1}; m_{01}\}, \{v_0, v_{i_2}; m_{02}\}, \{v_0, v_{i_3}; m_{03}\},\]
where at most two of $v_{i_1}, v_{i_2}$ and $v_{i_3}$ may be the same, then a reverse edge split will add one of the edges 
\[\{v_{i_1}, v_{i_2}; \bm_{02} - \bm_{01}\}, \{v_{i_2}, v_{i_3}; \bm_{03} - \bm_{02}\}, \{v_{i_3}, v_{i_1}; \bm_{01} - \bm_{03}\}.\]

\begin{verse}
\begin{figure}[h!]
\begin{center}
\begin{tikzpicture}[->,>=stealth,shorten >=1pt,auto,node distance=2.8cm,thick, font=\footnotesize] 
\tikzstyle{vertex1}=[circle, draw, fill=couch, inner sep=.5pt, minimum width=3.5pt, font=\footnotesize]; 
\tikzstyle{vertex2}=[circle, draw, fill=melon, inner sep=.5pt, minimum width=3.5pt, font=\footnotesize]; 
\tikzstyle{voltage} = [fill=white, inner sep = 0pt,  font=\scriptsize, anchor=center];

	\draw (0.1,0.05) circle (.9cm); 
	\fill[cloud] (0.1,0.05) circle (.9cm);
		
	\node[vertex1] (1) at (.2,.7)  {$1$};
	\node[vertex1] (2) at (.6,-.4) {$2$};
	   \node[vertex1] (4) at (-.4, -.1) {$3$};
	\node[vertex2] (3) at (1.7, .5) {$0$};

	\draw[thick] (3) --  node[voltage] {$\bm_{01}$} (1);
		\draw[thick] (3) -- node[voltage] {$\bm_{02}$} (2);
		\draw[thick] (3) -- node[voltage] {$\bm_{03}$} (4);

	%\path[bluey, very thick] (1.1, -0.5) edge [bend right] (5.6, -.7);
	\pgftransformxshift{2cm}
	\draw[very thick, bluey] (-.3, 0) -- (1, 0);
	\pgftransformxshift{2cm}

	\draw (0.1,0.05) circle (.9cm); 
	\fill[cloud] (0.1,0.05) circle (.9cm);
		
	\node[vertex1] (1) at (.2,.7)  {$1$};
	\node[vertex1] (2) at (.6,-.4) {$2$};
	   \node[vertex1] (4) at (-.4, -.1) {$3$};
	\draw[thick] (1) -- node[voltage, fill=cloud] {$ \bm$} (2);		
	
	\node[text width=2.5cm, text centered] at (2.5,0) { $\bm = \bm_{02} - \bm_{01}$};
	
\end{tikzpicture}
\caption{Reverse gain-preserving edge split. In this case the edge $\{v_{i_1}, v_{i_2}; \bm_{02} - \bm_{01}\}$ is added.  \label{fig:reverseEdgeSplit}}
\end{center}
\end{figure}
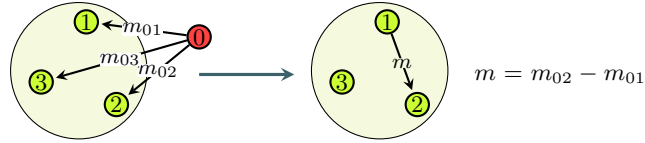\end{verse}

\begin{prop}[Reverse Gain-preserving Edge Split]
If a 3-valent vertex $v_0$ is deleted from a generically independent periodic orbit graph $\pogp$, then a single edge may be added between one pair of vertices formerly adjacent to $v_0$ so that the resulting graph $\pog$ is also a generically independent periodic orbit graph. 
\label{prop:reverseEdgeSplit}
\end{prop}

\begin{proof}
Suppose that the rows of $\R_0\pofwpr$ are independent for some realization $\p'$, and suppose that the vertex $v_0$ is connected to vertices $v_{i_1}, v_{i_2}$ and $v_{i_3}$, where at most two of these vertices are the same. Let $E^*$ be the edge set obtained from that of $G'$ by removing the edges incident to $v_0$. Let $\p$ be the realization obtained from $p'$ by deleting the point corresponding to $v_0$. Let $G_{12}, G_{23}$ and $G_{31}$ be the graphs with vertex set $V \backslash \{v_0\}$, and edge sets $E_{12} = E^* \cup \{ v_{i_1}, v_{i_2}; \bm_{02} - \bm_{01}\}$ and similarly for $E_{23}$ and $E_{31}$. If any of these graphs is independent at $\p$ then we are done. 

Assume to the contrary that no such graph is independent. Then the rows of the matrices corresponding to each of these frameworks are dependent. Writing $R_e$ as the row of the rigidity matrix corresponding to the edge $e$, we have 
\[\alpha_{12} R_{12} = \sum_{e \in E^*} -\alpha_e R_e \textrm{\ \ with \ } \alpha_{12} \neq 0,\]
\[\beta_{23} R_{23} = \sum_{e \in E^*} -\beta_e R_e \textrm{\ \ with \ } \beta_{23} \neq 0,\]
\[\gamma_{31} R_{31} = \sum_{e \in E^*} -\gamma_e R_e \textrm{\ \ with \ } \gamma_{31} \neq 0.\]
We now have two cases depending on whether the vertices $v_{i_1}, v_{i_2}$ and $v_{i_3}$ are distinct or not.
%\begin{enumerate}
%	\item The vertices $v_{i_1}, v_{i_2}$ and $v_{i_3}$ are distinct,
%	\item The vertices $v_{i_1}, v_{i_2}$ and $v_{i_3}$ are not distinct. 
%\end{enumerate}
%
If $v_{i_1}, v_{i_2}$ and $v_{i_3}$ are distinct, consider the graph on the vertices $\{v_0, v_{i_1}, v_{i_2}, v_{i_3}\}$ with all of the candidate edges (see Figure \ref{fig:proofOfReverseEdgeSplit}). This is $(2, 2)$-tight. Note that the net gain on any closed path in the graph is $(0,0)$, and hence this graph is $T$-gain equivalent to a graph with all gains identically zero. By Proposition \ref{prop:zeroGains} and Theorem \ref{thm:TgainsPreserveRigidity}, this graph is dependent. 

\begin{verse}
\begin{figure}[h!]
\begin{center}
\begin{tikzpicture}[->,>=stealth,shorten >=1pt,auto,node distance=2.8cm,thick, font=\footnotesize] 
\tikzstyle{vertex1}=[circle, draw, fill=couch, inner sep=.5pt, minimum width=3.5pt, font=\footnotesize]; 
\tikzstyle{vertex2}=[circle, draw, fill=melon, inner sep=.5pt, minimum width=3.5pt, font=\footnotesize]; 
\tikzstyle{voltage} = [fill=white, inner sep = 0pt,  font=\scriptsize, anchor=center];

	\node[vertex1] (1) at (0,1.5)  {$1$};
	\node[vertex1] (2) at (2,-.75) {$2$};
	   \node[vertex1] (4) at (-2, -.75) {$3$};
	\node[vertex2] (3) at (0, .2) {$0$};

	\draw[thick] (3) --  node[voltage] {$\bm_{01}$} (1);
	\draw[thick] (3) -- node[voltage] {$\bm_{02}$} (2);
	\draw[thick] (3) -- node[voltage] {$\bm_{03}$} (4);
\draw[thick] (1) --  node[voltage] {$\bm_{02} - \bm_{01}$} (2);
\draw[thick] (2) --  node[voltage] {$\bm_{03} - \bm_{02}$} (4);
\draw[thick] (4) --  node[voltage] {$\bm_{01} - \bm_{03}$} (1);
	
\end{tikzpicture}
\caption{This graph, corresponding to Case 1 of Proposition \ref{prop:reverseEdgeSplit}, is $(2,2)$-tight, and is $T$-gain equivalent to a graph with all zero gains, therefore a dependence exists among the edges.\label{fig:proofOfReverseEdgeSplit}}
\end{center}
\end{figure}
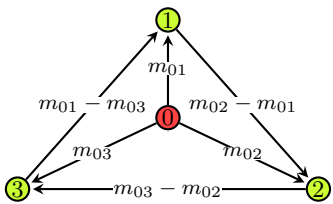\end{verse}

Therefore, we have \[\omega_{01}R_{01} + \omega_{02}R_{02} + \omega_{03}R_{03} + \omega_{12}R_{12} + \omega_{23}R_{23} + \omega_{31}R_{31} = 0.\]  
Scaling and substituting the expressions above, we obtain
\[\omega_{01}R_{01} + \omega_{02}R_{02} + \omega_{03}R_{03} +  \sum_{e \in E^*} -(\alpha_e' + \beta_e' + \gamma_e') R_e = 0.\]
At least one of $\omega_{01}, \omega_{02}, \omega_{03}$ must be non-zero (otherwise this is a dependence on a generic triangle), which is a dependence on the rows of $\R_0\pofwpr$, a contradiction. Therefore, at least one of the graphs $G_{12}, G_{23}, G_{31}$ must be independent. \\

If $v_{i_1}, v_{i_2}$ and $v_{i_3}$ are not all distinct, assume without loss of generality that $v_{i_2} = v_{i_3}$. We consider the graph on the vertices $\{v_0, v_{i_1}, v_{i_2}\}$ with all of the candidate edges (see Figure \ref{fig:proofOfReverseEdgeSplit2}). This graph has $|E| = 2|V| - 1$, and hence is dependent. The proof of this case now follows the proof of the previous case. Once again, at least one of $\omega_{01}, \omega_{02}, \omega_{03}$ must be non-zero, since we assumed $\p$ generic. That is, with $m_{02} \neq m_{03}$, the rows of the rigidity matrix corresponding to $\{v_1, v_2; m_{02} - m_{01}\}$ and $\{v_1, v_2; m_{03} - m_{01}\}$ are independent. \end{proof}
\begin{verse}
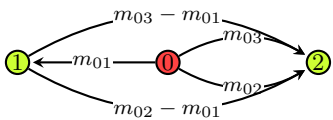
\begin{figure}[h!]
\begin{center}
\begin{tikzpicture}[->,>=stealth,shorten >=1pt,auto,node distance=2.8cm,thick, font=\footnotesize] 
\tikzstyle{vertex1}=[circle, draw, fill=couch, inner sep=.5pt, minimum width=3.5pt, font=\footnotesize]; 
\tikzstyle{vertex2}=[circle, draw, fill=melon, inner sep=.5pt, minimum width=3.5pt, font=\footnotesize]; 
\tikzstyle{voltage} = [fill=white, inner sep = 0pt,  font=\scriptsize, anchor=center];

	\node[vertex1] (1) at (-2,0)  {$1$};
	\node[vertex1] (2) at (2,0) {$2$};
	\node[vertex2] (0) at (0, 0) {$0$};

	\draw[thick] (0) -- node[voltage] {$\bm_{01}$} (1);
	
	\draw[thick] (0) edge  [bend right]  node[voltage] {$\bm_{02}$} (2);
	\draw[thick] (0) edge  [bend left] node[voltage] {$\bm_{03}$} (2);
\draw[thick] (1) edge [bend right]  node[voltage] {$\bm_{02} - \bm_{01}$} (2);
\draw[thick] (1) edge [bend left]  node[voltage] {$\bm_{03} - \bm_{01}$} (2);
%\draw[thick] (2) --  node[voltage] {$\bm_{03} - \bm_{02}$} (4);
%\draw[thick] (4) --  node[voltage] {$\bm_{01} - \bm_{03}$} (1);
	
\end{tikzpicture}
\caption{This graph, corresponding to Case 2 of Proposition \ref{prop:reverseEdgeSplit}, satisfies $|E| = 2|V| - 1$, therefore a dependence exists among the edges. \label{fig:proofOfReverseEdgeSplit2}}
\end{center}
\end{figure}\end{verse}
%

%Both the vertex addition and the edge split preserve the relationship between the number of edges and the number of vertices in the $d$-periodic orbit graph. %If $|E| = 2|V| - 2$ then $|E'| = 2|V'| - 2$ as well. 

The process of deleting a three-valent vertex from $\pog$ by a reverse edge split, and then performing an edge split will not usually produce a graph that is identical to the original (see Figure \ref{fig:deletingSplitting}). However, we can ensure that we always produce a graph whose rigidity matrix has the same rank as the original, using the following lemma:

\begin{lem}
Let $\pog$ be a periodic orbit graph, and let $\pogp$ be a reverse edge split of $\pog$. Then for some edge split $\langle \overline G, \overline m \rangle$ of $\pogp$ with $G = \overline G$, the resulting graph $\pogom$ is $T$-gain equivalent to $\pog$. 
%
%Let $\pogom$ be the graph obtained from $\pog$ by a reverse edge split, followed by an edge split on the added edge. Then $\pog$ and $\pogom$ are $T$-gain equivalent for some spanning tree $T$. 
\label{lem:deletingSplitting}
\end{lem}

\begin{proof}
Let $v_0$ be a $3$-valent vertex of $\pog$, adjacent to vertices $v_1, v_2, v_3$ (see Figure \ref{fig:deletingSplitting}). After deleting $v_0$, suppose without loss of generality that the edge $e = \{v_1, v_2; \bm_{02} - \bm_{01}\}$ was added to form the graph $\pogp$. We perform an edge split on this edge to obtain a graph that differs from our original orbit graph, but which has a  rigidity matrix with the same rank. In particular, we add to $\pogp$ the vertex $v_0$ and the three edges:
\[\{v_0, v_1; (0, 0)\}, \ \ \{v_0, v_2; \bm_{02} - \bm_{01}\}, \ \ \{v_0, v_3; \bm_{03} - \bm_{01}\}.\]
Let the resulting infinitesimally rigid graph be denoted $\pogom$. Note that the gains on the first two edges are determined by the reverse edge split, but the gain on the third edge is a `free' choice.

\begin{figure}
\begin{center}
\begin{tikzpicture}[->,>=stealth,shorten >=1pt,auto,node distance=2.8cm,thick, font=\footnotesize, scale=1.2] 
\tikzstyle{vertex1}=[circle, draw, fill=couch, inner sep=.5pt, minimum width=3.5pt, font=\footnotesize]; 
\tikzstyle{vertex2}=[circle, draw, fill=melon, inner sep=.5pt, minimum width=3.5pt, font=\footnotesize]; 
\tikzstyle{voltage} = [fill=white, inner sep = 0pt,  font=\scriptsize, anchor=center];

	\draw[thin] (0.1,0.05) circle (.9cm); 
	\fill[cloud] (0.1,0.05) circle (.9cm);
		
	\node[vertex1] (1) at (.2,.7)  {$1$};
	\node[vertex1] (2) at (.3,-.55) {$2$};
	   \node[vertex1] (3) at (-.5, 0) {$3$};
	\node[vertex2] (0) at (1.7, 0) {$0$};
	
	 \node at (0.1, -1.1) {$\pog$};

	\draw[thick] (0) --  node[voltage] {$\bm_{01}$} (1);
		\draw[thick] (0) -- node[voltage] {$\bm_{02}$} (2);
		\draw[thick] (0) -- node[voltage] {$\bm_{03}$} (3);
		
		\pgftransformxshift{4cm}	
		\draw[thin] (0.1,0.05) circle (.9cm); 
	\fill[cloud] (0.1,0.05) circle (.9cm);
		
	\node[vertex1] (1) at (.2,.7)  {$1$};
	\node[vertex1] (2) at (.3,-.55) {$2$};
	   \node[vertex1] (3) at (-.5, 0) {$3$};
	   
	   \node at (0.1, -1.1) {$\pogp$};
	   
	   \draw[thick] (1) -- node[voltage] {$\bm_{02}- \bm_{01}$} (2);
	   
	   \pgftransformxshift{3cm}	
	   \draw[thin] (0.1,0.05) circle (.9cm); 
	   	\fill[cloud] (0.1,0.05) circle (.9cm);

	\node[vertex1] (1) at (.2,.7)  {$1$};
	\node[vertex1] (2) at (.3,-.55) {$2$};
	   \node[vertex1] (3) at (-.5, 0) {$3$};
	\node[vertex2] (0) at (1.7, 0) {$0$};
	
		   \node at (0.1, -1.1) {$\pogo$};

	\draw[thick] (0) --  node[voltage] {$(0,0)$} (1);
		\draw[thick] (0) -- node[voltage] {$\bm_{02}-\bm_{01}$} (2);
		\draw[thick] (0) -- node[voltage] {$\bm_{03}-\bm_{01}$} (3);

	\end{tikzpicture}
	\caption{Proof of Lemma \ref{lem:deletingSplitting}:  Deleting a $3$-valent vertex from $\pog$, followed by an edge split, results in a $T$-gain equivalent periodic orbit graph $\pogo$. \label{fig:deletingSplitting}}
	\end{center}
\end{figure}
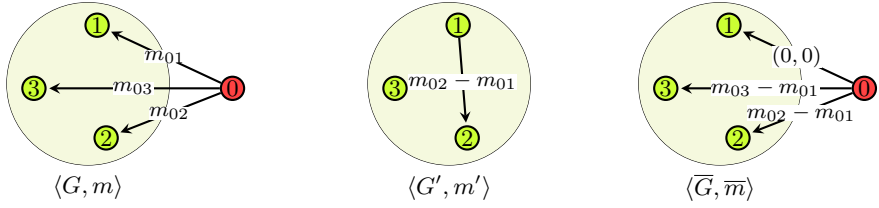	

%It is clear from the figure that any cycle passing through vertex $0$ has the same contribution to the net gain from the edges of $\pog$ or $\pogo$. Therefore, $\pog$ and $\pogo$ are $T$-gain equivalent. \nts{check this}
	
	Now let $T'$ be a spanning tree in $G'$ with root $u = v_1$ that does {\it not} include the edge $e = \{v_1, v_2\}$ (which has gain $\bm_{02} - \bm_{01}$ in $\pog'$). It is always possible to select such a tree, since deleting this edge will not disconnect the graph. Let $T$ be the spanning tree of $G$ created by adding the edge $\{v_0, v_1\}$ to $T'$. This edge has gain  $\bm_{01}$ in $\pog$, and gain $(0,0)$ in $\pogom$.  Performing the $T$-gain procedure on $\pog$ and $\pogom$ with $T$, we obtain identical periodic orbit graphs. For example, the edge $e_2 = \{v_0, v_2, \bm_{02}\} \in \pog$ has $T$-gain 
\begin{eqnarray*} \bm_{T}(e_2) 	& = & \bm(v_0, T) + \bm_{02} - \bm(v_2, T) \\
							& = & -\bm_{01} + \bm_{02} - \bm(v_2, T)  \\
							& = & (0,0) + (\bm_{02} - \bm_{01}) - \bm(v_2, T) \\
							& = & (0,0) + (\bm_{02} - \bm_{01}) - \overline \bm(v_2, \overline T) \\
							& = & \overline \bm_{T}(e_2). 
\end{eqnarray*}
The same is true of the other edges added in the edge split, and since $T$ = $\overline T$ for all of the edges of $\pogp$, the orbit graphs are $T$-gain equivalent. That is, 
\[\langle G, \bm_T \rangle = \langle G, \overline \bm_{T} \rangle.\]
\end{proof}
\index{inductive constructions!on $\Tor^2$|)}

%%%%Periodic Henneberg Theorem%%%%
\subsection{Periodic Henneberg Theorem}
\label{sec:periodicHenneberg}
\index{Henneberg's Theorem!periodic $\Tor^2$} \index{periodic Henneberg Theorem}
\begin{thm}[Periodic Henneberg Theorem]
A periodic orbit graph $\pog$  is generically minimally rigid on $\Tor^2$ if and only if it can be constructed from a single vertex on $\Tor^2$ by a sequence of gain-preserving vertex additions and edge splits. 
\label{thm:Henneberg}
\end{thm}

The proof of this result follows the proof of Henneberg's result appearing in, for example, \cite{CountingFrameworks}. 
\begin{proof}

$(\Longleftarrow)$ Let $\pog$ be the periodic orbit graph consisting of a single vertex, $V = \{v_0\}$ and $E = \emptyset$. This is trivially a generically rigid periodic orbit graph. By Lemmas \ref{prop:vertexAddition} and \ref{prop:edgeSplit} we can perform gain-preserving vertex additions and edge splits to obtain a new generically rigid periodic orbit graph. Since both operations preserve the count $|E| = 2|V| -2$, the new orbit graph is generically minimally rigid on $\Tor^2$.

$(\Longrightarrow)$ This direction is proved by induction on the number of vertices, $|V|$. As noted above the single vertex on $\Tor^2$ is generically infinitesimally rigid, which provides the base case. 

Consider a generically minimally rigid periodic orbit graph $\pog$ with $|V|\geq2$, and assume that all infinitesimally rigid frameworks on $\Tor^2$ with fewer than $|V|$ vertices satisfy the hypothesis. Since $|E| = 2|V| - 2$, the average valence of any given vertex is 
$$\rho = \frac{2|E|}{|V|} = \frac{2(2|V| - 2)}{|V|} = \frac{4|V|-4}{|V|} = \ 4 - 4/|V|\  <\  4.$$
In addition, because the orbit graph is infinitesimally rigid on $\Tor^2$, every vertex has valence at least $2$ (any graph with a pendent vertex is {\it not} infinitesimally rigid). These two facts together imply that $G$ must have a vertex of valence either $2$ or $3$. 

If $G$ has a vertex of valence $2$, then $\pog$ is a (periodic) vertex-addtion of an infinitesimally rigid framework on a graph $(V', E')$ by Proposition \ref{prop:vertexAddition}. 

If $G$ has a vertex of valence $3$, then by Lemma \ref{lem:deletingSplitting} $\pog$ is $T$-gain equivalent to an edge split of an infinitesimally rigid framework on a graph $(V'', E'')$. In either case, $|V'| = |V''| < n$, and we may apply the induction hypothesis to the underlying graph. %
\end{proof}

For a periodic orbit graph $\pog$, we call the sequence of orbit graphs 
\[\langle G_1, m_1 \rangle, \langle G_2, m_2 \rangle, \dots, \langle G_n, m_n \rangle = \pog\]
beginning with a single vertex $|V_1| = 1$ and ending with $\pog$ ($|V_n| = n = |V|$) the {\it (periodic) Henneberg sequence} for $\pog$.  An example of a Henneberg sequence is shown in Figure \ref{fig:henn}. 

\begin{verse}
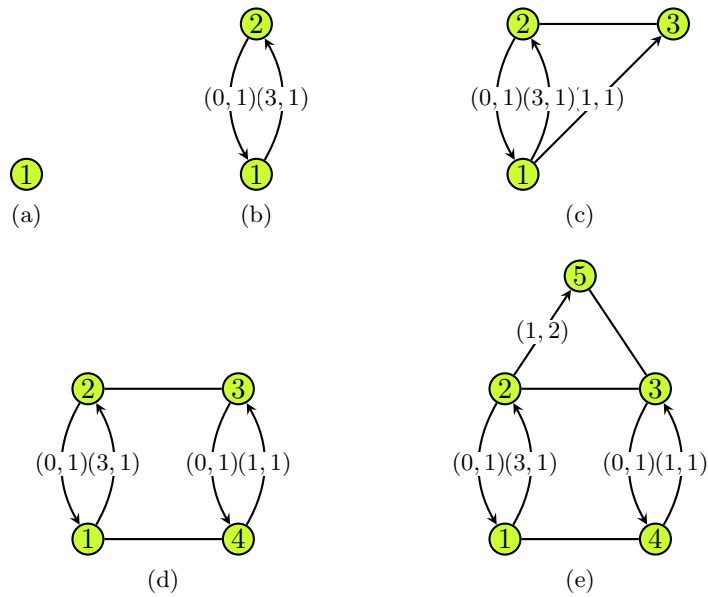
\begin{figure}
\begin{center}
\subfloat[]{\begin{tikzpicture}[auto,node distance=2cm, thick] 
\tikzstyle{vertex1}=[circle, draw, fill=couch, inner sep=1pt, minimum width=5pt]; 
\tikzstyle{vertex2}=[circle, draw, fill=melon, inner sep=1pt, minimum width=5pt]; 
\tikzstyle{gain} = [fill=white, inner sep = 0pt,  font=\footnotesize, anchor=center];

%\node[vertex1] (1) {$2$}; 
%\node[vertex1] (2) [right of=1] {$3$}; 
%\node[vertex1] (3) [below of=2] {$4$}; 
\node[vertex1] (4) [below of=1] {$1$}; 
%\node[vertex1] (5) at (1, 1.5) {$5$};

\end{tikzpicture}}\hspace{2cm}
\subfloat[]{\begin{tikzpicture}[auto,node distance=2cm, thick] 
\tikzstyle{vertex1}=[circle, draw, fill=couch, inner sep=1pt, minimum width=5pt]; 
\tikzstyle{gain} = [fill=white, inner sep = 0pt,  font=\footnotesize, anchor=center];

\node[vertex1] (1) {$2$}; 
%\node[vertex1] (2) [right of=1] {$3$}; 
%\node[vertex1] (3) [below of=2] {$4$}; 
\node[vertex1] (4) [below of=1] {$1$}; 

\pgfsetarrowsend{stealth}
\path 
(4) edge [bend right] node[gain] {$(3,1)$}  (1)
(1) edge [bend right] node[gain] {$(0,1)$} (4);
\pgfsetarrowsend{}

\end{tikzpicture}}\hspace{2cm}
\subfloat[]{\begin{tikzpicture}[auto,node distance=2cm, thick] 
\tikzstyle{vertex1}=[circle, draw, fill=couch, inner sep=1pt, minimum width=5pt]; 
\tikzstyle{gain} = [fill=white, inner sep = 0pt,  font=\footnotesize, anchor=center];

\node[vertex1] (1) {$2$}; 
\node[vertex1] (2) [right of=1] {$3$}; 
%\node[vertex1] (3) [below of=2] {$4$}; 
\node[vertex1] (4) [below of=1] {$1$}; 

\path (1) edge (2);

\pgfsetarrowsend{stealth}
\path 
(4) edge node[gain] {$(1,1)$} (2)
(4) edge [bend right] node[gain] {$(3,1)$}  (1)
(1) edge [bend right] node[gain] {$(0,1)$} (4);
\pgfsetarrowsend{}

\end{tikzpicture}}\\
\subfloat[]{\begin{tikzpicture}[auto,node distance=2cm, thick] 
\tikzstyle{vertex1}=[circle, draw, fill=couch, inner sep=1pt, minimum width=5pt]; 
\tikzstyle{gain} = [fill=white, inner sep = 0pt,  font=\footnotesize, anchor=center];

\node[vertex1] (1) {$2$}; 
\node[vertex1] (2) [right of=1] {$3$}; 
\node[vertex1] (3) [below of=2] {$4$}; 
\node[vertex1] (4) [below of=1] {$1$}; 
%\node[vertex1] (5) at (1, 1.5) {$5$};

\path 
%(5) edge (2)
(1) edge  (2) 
(4) edge (3);

\pgfsetarrowsend{stealth}
\path 
%(1) edge node[gain] {$(1,2)$} (5) 
(2) edge [bend right] node[gain] {$(0,1)$} (3)
(4) edge [bend right] node[gain] {$(3,1)$}  (1)
(1) edge [bend right] node[gain] {$(0,1)$} (4)
%(4) edge node[gain]{$(1,1)$} (3); 
(3) edge [bend right] node[gain]{$(1,1)$}  (2);

\pgfsetarrowsend{}

\end{tikzpicture}}\hspace{2cm}
\subfloat[]{\begin{tikzpicture}[auto,node distance=2cm, thick] 
\tikzstyle{vertex1}=[circle, draw, fill=couch, inner sep=1pt, minimum width=5pt]; 
\tikzstyle{gain} = [fill=white, inner sep = 0pt,  font=\footnotesize, anchor=center];

\node[vertex1] (1) {$2$}; 
\node[vertex1] (2) [right of=1] {$3$}; 
\node[vertex1] (3) [below of=2] {$4$}; 
\node[vertex1] (4) [below of=1] {$1$}; 
\node[vertex1] (5) at (1, 1.5) {$5$};

\path 
(5) edge (2)
(1) edge  (2) 
(4) edge (3);
%(3) edge [bend right]   (2);

\pgfsetarrowsend{stealth}
\path 
(1) edge node[gain] {$(1,2)$} (5) 
(2) edge [bend right] node[gain] {$(0,1)$} (3)
(4) edge [bend right] node[gain] {$(3,1)$}  (1)
(1) edge [bend right] node[gain] {$(0,1)$} (4)
(3) edge [bend right] node[gain]{$(1,1)$}  (2);
%(4) edge node[gain]{$(1,1)$} (3); 
\pgfsetarrowsend{}
\end{tikzpicture}}

\caption{An example of a periodic Henneberg sequence. The single vertex (a) becomes a single cycle through a vertex addition (b). Adding a third vertex in (c), then splitting off the edge $\{v_1, v_3; (1,1)\}$ and adding the fourth vertex (d). The final graph is shown in (e). \label{fig:henn}}
\end{center}
\end{figure}\end{verse}

%%%% COMBINATORIAL STUFF

\section{$(2, 2)$-tight graphs}
\label{sec:combinatorial}

By Theorem \ref{thm:perMaxwell}, generically minimally rigid periodic orbit graphs on $\Tor^2$ are $(2,2)$-tight. In this section we outline some combinatorial results about $(2, 2)$-tight graphs, starting with some well known properties. These results will be used to prove our main result in the subsequent section.
\begin{lem}
Let $G$ be a $(2,2)$-tight graph.
Let $v_0$ be some vertex of the graph. Let $\mathcal G$ be the set of all subgraphs $G' \subseteq G$ that contain $v_0$ and satisfy $|E'| = 2|V'| - 2$. Then $\mathcal G$ is a lattice. 
\label{lem:lattice}
\end{lem}
This property is a consequence of the structure theorem for $(k, \ell)$-sparse graphs that can be found in \cite{pebbleGameSparse}. Another property from the same paper (\cite{pebbleGameSparse}, Theorem 5) is the following: 
\begin{lem}
Let $G=(V, E)$ be a $(2,3)$-tight graph. Let $G_1 = (V_1, E_1)$ and  $G_2 = (V_2, E_2)$ be $(2,3)$-tight vertex-induced subgraphs of $G$. If  $V_1 \cap V_2 \neq 0$, then either
\begin{enumerate}
\item $|V_{1} \cap V_{2}| = 1$ and $|E_{1} \cup E_{2}| = 2|V_{1} \cup V_{2}| - 4$ {\rm or}
\item $|V_{1} \cap V_{2}| > 1$ and $|E_{1} \cup E_{2}| = 2|V_{1} \cup V_{2}| - 3$ and  $|E_{1} \cap E_{2}| = 2|V_{1} \cap V_{2}| - 3$.
\end{enumerate}
\label{lem:intersection}
\end{lem}

\begin{cor}
The graph $( V_{1} \cap V_{2}, E_{1} \cap E_{2})$ is connected.
\label{cor:intersection}
\end{cor}

The next result collects some further useful properties of $(2,2)$-tight graphs.  

\begin{lem}
Let $G = (V, E)$ be a $(2,2)$-tight graph, and let $v_0 \in V$ have $N(v_0) = \{v_1, v_2, v_3\}$ (see Figure \ref{fig:subgraphs}). Let $H$ be the graph obtained from $G$ by deleting vertex $v_0$ and its incident edges. Then:
\begin{enumerate}[(i)]
\item $v_1, v_2$ and $v_3$ are not all in a $(2,2)$-tight subgraph of $H$. 
\item If $v_1, v_2$ are in a $(2,2)$-tight subgraph $G_{12}$ of $H$, then neither pairs of vertices $v_1, v_3$ or $v_2, v_3$ are in the span of a $(2,2)$-tight subgraph of $H$. 
\item If $v_1, v_2$ are in a $(2,2)$-tight subgraph $G_{12}$ of $H$, and $v_1, v_3$ are in a $(2,3)$-tight subgraph $G_{13}$ of $H$, then $V_{12} \cap V_{13} = \{v_1\}$.
\item If $v_1$ and $v_2$ are in a $(2,2)$-tight subgraph $G_{12}$ of $H$, then the pairs of vertices $v_1, v_3$ and $v_2, v_3$ are not both in the span of $(2,3)$-tight subgraphs of $H$. 
\item If some set of edge $E' \subset E$ spans $v_1, v_2, v_3$ and is $(2,3)$-tight, then the set of edges $E'$ together with their incident vertices $V'$ form a vertex-induced subgraph of $H$. 
\end{enumerate}
\label{lem:ref1}
\end{lem}

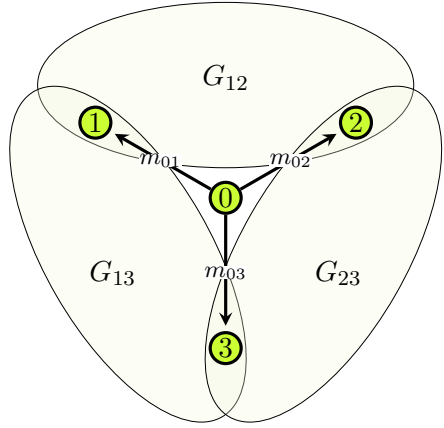
\begin{figure}
\begin{center}
\begin{tikzpicture}[very thick,scale=1,>=stealth,->,shorten >=2pt,looseness=0.5,auto] 

\tikzstyle{vertex1}=[circle, draw, fill=couch, inner sep=1pt, minimum width=5pt]; 
\tikzstyle{vertex2}=[circle, draw, fill=melon, inner sep=1pt, minimum width=5pt]; 
\tikzstyle{gain} = [fill=white, inner sep = 0pt,  font=\footnotesize, anchor=center];
\tikzstyle{arrow}=[->,shorten >=1pt,>=stealthÕ,semithick]  
\pgfsetfillopacity{0.3} 
\fill[cloud] (90:1.5cm) ellipse (25mm and 11mm); 
\draw[very thin] (90:1.5cm) ellipse (25mm and 11mm); 
\fill[cloud, rotate=120] (90:1.5cm) ellipse (25mm and 11mm); 
\draw[ very thin, rotate=120] (90:1.5cm) ellipse (25mm and 11mm); 
\fill[cloud, rotate=240] (90:1.5cm) ellipse (25mm and 11mm); 
\draw[very thin, rotate=240] (90:1.5cm) ellipse (25mm and 11mm); 
\pgfsetfillopacity{1}

	\node[vertex1] (x3) at (-90:2cm){$3$};
	\node[vertex1] (x1) at (-210:2cm)  { $1$};
	\node[vertex1] (x2) at (30:2cm) {$2$};
	\node[vertex1] (x) at (0,0) {$0$};
	
	\draw[->] (x) -- node[gain] {  $ \bm_{01}$} (x1) 	;
	\draw[->] (x) --  node[gain] {  $ \bm_{02}$} (x2);
	\draw[->] (x) --  node[gain] {  $ \bm_{03}$} (x3);

\node at (1.5, -1) {$G_{23}$};
\node at (0, 1.6) {$G_{12}$};
\node at (-1.5, -1) {$G_{13}$};	

\end{tikzpicture}
\caption{Lemma \ref{lem:ref1}: Subgraphs of $H$ containing pairs of vertices from $\{v_1, v_2, v_3\}$. \label{fig:subgraphs} }
\end{center}
\end{figure}

\begin{proof}
{\it (i)} If such a subgraph existed, the addition of $v_0$ with its three incident edges would be an over-counted subgraph of $G$, which contradicts the assumption that $G$ is $(2,2)$-tight. 

{\it (ii)} Suppose, toward a contradiction, that the pair of vertices $v_2, v_3$ are contained in the $(2,2)$-tight subgraph $G_{23}$.  Then both $G_{12}$ and $G_{23}$ are members of the lattice of subgraphs containing the vertex $V_2$. By Lemma \ref{lem:lattice} it follows that 
\[|E_{12} \cup E_{23}| = 2|V_{12} \cup V_{23}| - 2.\]
Hence $(V_{12} \cup V_{23}, E_{12} \cup E_{23})$ is a subgraph containing all three vertices $v_1, v_2, v_3$ but not $v_0$, and by the proof of part {\it (i)}, this is a contradiction. 

{\it (iii)} First consider the intersection of $G_{12}$ and $G_{13}$. We find
\begin{eqnarray}
|E_{12} \cup E_{13}|  + |E_{12} \cap E_{13}| & = & |E_{12}| + |E_{13}|\nonumber \\
& = & (2|V_{12}| - 2) +(2|V_{13}| - 3) \nonumber \\
& = & 2(|V_{12}|+|V_{13}|) - 5\nonumber \\
& = & 2| V_{12} \cup V_{13}| - 2 + 2|V_{12}\cap V_{13}| - 3. \label{eqn:intersection2}
\end{eqnarray}

Toward a contradiction, suppose that $|V_{12} \cap V_{13}| > 1$. 
Then, because $G_{13}$ is $(2,3)$-tight, we have $|E_{12} \cap E_{13}| \leq 2|V_{12} \cap V_{13}| - 3$, and hence (\ref{eqn:intersection2}) becomes $|E_{12}\cup E_{13}|  \geq  2|V_{12}\cup V_{13}| - 2$.
In fact, since the reverse inequality always holds, we have equality
$|E_{12} \cup E_{13}|  =  2|V_{12}\cup V_{13}| - 2$, which is a contradiction by part {\it (i)}. 

{\it (iv)} 
Suppose, toward a contradiction, that $G_{23}$ and $G_{13}$ are $(2,3)$-tight subgraphs of $H$ which contain the vertices  $v_2, v_3$ and $v_1, v_3$  respectively. 

By {\it (iii)}, it must be the case that $|V_{12} \cap V_{13}| = |V_{12} \cap V_{23}| = 1$.
Then $|E_{12} \cap E_{k}| = 0$, where $k \in \{13, 23\}$ and hence $|E_{12} \cup E_k| = 2|V_{12} \cup C_k| - 3$.

Let $\overline{G} \subset G$ be the vertex-induced subgraph of $G$ on the vertices $V_{23} \cup V_{31}$. Then $|\overline{E}| \geq |E_{23} \cup E_{31}|$. %, since there could be edges in $\overline E$ that were not part of either $E_{23}$ or $E_{31}$, but that connect vertices in $V_{23} \cup V_{31}$. 
By Lemma \ref{lem:intersection}, we know that either

{\it Case 1.} $|V_{23} \cap V_{31}| = 1$ and $|E_{23} \cup E_{31}| = 2|V_{23} \cup V_{31}| - 4$ {\it or}

{\it Case 2.} $|V_{23} \cap V_{31}| > 1$ and $|E_{23} \cup E_{31}| = 2|V_{23} \cup V_{31}| - 3.$

\noindent We will deal with {\it Case 2} first. In this case $|\overline E| = |E_{23} \cup E_{31}| = 2|V_{23} \cup V_{31}| - 3$ (i.e. $\overline{G}$ is a vertex-induced subgraph), since $v_1, v_2, v_3 \in \overline V$ so part {\it (i)} applies. 
%In other words, there can't be any edges in $\overline E$ that aren't also in $E_{23} \cup E_{31}$, otherwise the graph $(\overline V, \overline E)$ would be over-counted. 
We now consider the intersection graph $\overline G \cap G_{12} = (\overline V \cap V_{12}, \overline E \cap E_{12})$. 
But $|\overline V \cap V_{12}| > 1$, which is a contradiction by the argument of part {\it (iii)}. 

We now return to {\it Case 1.} Here $|\overline E| \geq |E_{23} \cup E_{31}| = 2|V_{23} \cup V_{31}| - 4$. If $|\overline E| > |E_{23} \cup E_{31}|$, then we are in the situation of {\it Case 2}. Hence we may assume that $|\overline E| = |E_{23} \cup E_{31}|$, and is therefore a vertex-induced subgraph.  

Notice that the intersection of $V_{12}$ with either of $V_{23}$ or $V_{31}$ may only consist of one element by the argument of part {\it (iii)}. So the three subgraphs must intersect pair-wise in one of the vertices $v_1, v_2, v_3$, and it follows that the intersection of all three of these subgraphs is empty. 
\begin{eqnarray}
|\overline E \cup E_{12}|  + |\overline E \cap E_{12}| & = & |\overline E| + |E_{12}|\nonumber \\
& = & (2|\overline V| - 4) + (2|V_{12}| - 2)\nonumber \\
& = & 2(|\overline V|+|V_{12}|) - 6\nonumber \\
& = & (2|\overline V \cup V_{12}| - 3) + (2|\overline V \cap V_{12}| - 3). \label{eqn:intersection3}
\end{eqnarray}

But we know that $|\overline V \cap V_{12}| = 2$, and it must be the case that $|\overline E \cap E_{12}| = 0$, since the intersection of the three graphs is empty.
Hence equation (\ref{eqn:intersection3}) becomes $|\overline E \cup E_{12}| = 2|\overline V \cup V_{12}| - 2$ which is a contradiction by {\it (i)}, since $v_1, v_2, v_3 \in \overline V \cup V_{12}$. Adding $v_0$ would violate the subgraph property of $G$, and this concludes the proof of part {\it (iv)}.  %

{\it (v)} Follows from {\it (i)}.
\end{proof}
\section{Gain assignments determine rigidity on $\Tor^2$}
\label{sec:gainAssignmentsDetermineRigidity}

In this section, we characterize the generic rigidity properties of a framework on the two-dimensional fixed torus $\Tor^2$ by its gain assignment. In Section \ref{sec:constructiveGainAssignments} we show that only graphs with {\it constructive gain assignments} can be rigid, and Section \ref{sec:constructiveGainAssignmentsAreSufficient} will demonstrate that all such periodic orbit graphs are generically rigid. This forms a fixed torus version of Laman's theorem, which is a second main result.

%%% constructive gain assignments
\subsection{Constructive gain assignments}
\label{sec:constructiveGainAssignments}

Let $\pog$ be a periodic orbit graph. Let $C$ be a closed oriented cycle with no repeated vertices, starting and ending at a vertex $u$ in $G$. Recall that the {\it net (cycle) gain} is the sum $m_{C}$ of the gain assignments of the edges of the cycle, where the signs of the edges are determined by the traversal direction specified by the orientation. We say the net gain on the cycle is {\it non-zero} or {\it non-trivial} if it is non-zero on at least one of the two coordinates of $m_C \in \mathbb Z^2$.  We similarly define the {\it net (path) gain} on a path to be the sum of the gains on the edges of the path. 

Let $\pog$ be a periodic orbit graph where $G$ is $(2,2)$-tight.  A {\it constructive gain assignment} on $G$ is a map $\bm: E^+ \rightarrow \mathbb Z^2$ such that every subgraph $G' \subset G$ with $G' = (V', E')$ and $|E'| = 2|V'| - 2$ contains some cycle with a non-zero net gain. A cycle $C$ with a non-zero net gain will be called a {\it constructive cycle}. If $\langle H, m_H \rangle$ is a graph with $|E(H)| > 2|V(H)| - 2$, we say that $\langle H, m_H \rangle$ has a constructive gain assignment if there is some subgraph $G \subset H$ such that $m_H |_G$ is constructive on $G$.

\begin{prop}
Let $\pog$ be a periodic orbit graph where $G$ is $(2,2)$-tight. If $\pofw$ is infinitesimally rigid on $\Tor^2$ for some realization $\p$, then $\bm$ is constructive.
\label{prop:constructiveIsNecessary}
\end{prop}

\begin{proof}
We will show the contrapositive. Suppose that $\bm$ is not constructive, and therefore there exists a subgraph $\pogp \subseteq \pog$ with $|E'| = 2|V'| - 2$ and no constructive cycles. Let $T'$ be a spanning tree in $G'$, and expand $T'$ to a spanning tree $T$ of all of $G$. This is always possible, since $\pofw$ is infinitesimally rigid on $\Tor^2$, and therefore $G$ is connected.

Perform the $T$-gain procedure on $\pog$. Every edge in $T$ and therefore in $T'$ will have zero gains, and hence no other edge in $E'$ may have non-zero gain, since the $T$-gain procedure preserves net cycle gains. 

Hence $\pogp$ consists of $2|V'| - 2$ edges with zero gains, which correspond to dependent rows in the rigidity matrix, since at most $2|V'| - 3$ edges without gains can be independent in the rigidity matrix, by Lemma \ref{prop:zeroGains}. Therefore, 
$$\rank \R_0\pofw < 2|V| - 2,$$
and $\pofw$ is infinitesimally flexible on $\Tor^2$.%
\end{proof}

For simplicity in what follows, we say that a periodic orbit graph $\pog$ where $G$ is $(2,2)$-tight and $m$ is a constructive gain assignment is a {\it constructive periodic orbit graph}, or that the periodic orbit graph is {\it constructive}. For general gain graphs, cycles with a non-zero net gain are usually called {\it unbalanced} cycles. We use the term `constructive' for this special class of $\mathbb Z^2$ -labeled gain graphs to connote the idea that these cycles {\it construct} connectivity in the infinite periodic frameworks.

The following section will demonstrate that constructive gain assignments are also sufficient for infinitesimal rigidity on $\Tor^2$. 

%%%%%Main Result%%%%%%
\subsection{Periodic Laman Theorem on $\Tor^2$}
\label{sec:constructiveGainAssignmentsAreSufficient}

The following is the second main result, which completely characterizes the generic minimal rigidity of periodic orbit frameworks on $\Tor^2$. The proof is developed through Propositions \ref{prop:canAlwaysDelete}, \ref{prop:v0is2Valent} and \ref{prop:twoConstFailure}.

\begin{thm}[Periodic Laman Theorem]
Let $\pog$ be a periodic orbit graph. Then $\pog$ is generically minimally rigid on $\Tor^2$ if and only if $\pog$ is a constructive periodic orbit graph.
\label{thm:perLaman}
\end{thm}

\begin{proof}
We will show that $\pog$ is a constructive periodic orbit graph if and only if $\pog$ can be constructed from a single vertex by a sequence of gain-preserving Henneberg moves. 

Suppose that $\pog$ has been constructed by a sequence of gain-preserving Henneberg moves. Then by the Periodic Henneberg Theorem (Theorem \ref{thm:Henneberg}), $\pog$ is generically minimally rigid on $\Tor^2$. By Theorem \ref{thm:perMaxwell}, $G$ is $(2,2)$-tight, and by Proposition \ref{prop:constructiveIsNecessary}, $m$ is a constructive gain assignment, hence $\pog$ is constructive.  

The `only if' part of the proof proceeds by induction on the number of vertices, $n = |V|$. 

First note that the hypothesis is true in the case $|V|=|E|=2$. By the proof of the Periodic Henneberg Theorem (Theorem \ref{thm:Henneberg}), any periodic orbit graph $\pog$ with a constructive gain assignment with $2$ vertices can be obtained as a vertex addition on a single vertex (which is minimally rigid on $\Tor^2$). 

For the inductive step, let $\pog$ be a constructive periodic orbit graph with  $|V| = n \geq 3$, and we assume the claim holds for any constructive periodic orbit graph with $|V|<n$. By Proposition \ref{prop:canAlwaysDelete} we can always delete a $2$- or $3$-valent vertex such that the resulting periodic orbit graph $\pogp$ is constructive. Then $|V'| = n - 1$, hence the inductive hypothesis applies, and $\pogp$ is generically minimally rigid on $\Tor^2$. 

To obtain the original orbit graph under consideration, $\pog$, we perform the appropriate periodic Henneberg move on the graph $\pogp $ as follows:
\begin{enumerate}
\item If a $2$-valent vertex was deleted, simply add back the same edges that were deleted. 
\item If a $3$-valent vertex was deleted, then by Lemma \ref{lem:deletingSplitting} we can edge split the added edge to a obtain the orbit graph $\pogom$, which is $T$-gain equivalent to $\pog$. %We can reverse the $T$-gain procedure to obtain $\pog$ itself. 
\end{enumerate}

In either case, $\pog$ is generically minimally rigid on $\Tor^2$. In the second case, Theorem \ref{thm:TgainsPreserveRigidity} applies to show that $\pog$ is minimally rigid because $\pogom$ is minimally rigid. %
\end{proof}

\begin{prop}
Let $\pog$ be a constructive periodic orbit graph. Then it is always possible to delete any $2$-valent vertex $v_0$, or perform a reverse edge split on any $3$-valent vertex $v_0$ such that the resulting periodic orbit graph $\langle G_0, \bm_0 \rangle$ is also a constructive periodic orbit graph.
\label{prop:canAlwaysDelete}
\end{prop}
%\ER{ref 2 says: merge this with Periodic Henneberg}

%\ER{deleted the statements of prop:technical2valent (formerly prop 5.8) and prop:technical3valent (formerly prop 5.9}
\begin{proof}

Deleting a $2$-valent vertex $v_0$ leaves a graph $G'$ which is a subgraph of the original graph $G$ with $|E'| = 2|V'| - 2$. Since $\bm$ was constructive, this subgraph $\pogp$ also has a constructive gain assignment. %

If $v_0$ is $3$-valent, we have two cases, either $v_0$ is adjacent to two distinct vertices, or $v_0$ is adjacent to three distinct vertices.

Suppose first that $v_0$ is adjacent to two distinct vertices $v_1$ and $v_2$, and that there are two copies of the edge connecting $v_0$ to $v_1$, with gain assignments $\bm_a$ and $\bm_b$. Let the gain assignment of the edge connecting $v_0$ and $v_2$ be $\bm_{02}$ Then the two candidates for edges to insert are $\{v_1, v_2; \bm_{02}- \bm_a\}$, or $\{v_1, v_2; \bm_{02} - \bm_b\}$ (see Figure \ref{fig:schematic}(a)).
By Proposition \ref{prop:v0is2Valent}, it is always possible to add one of these two candidate edges, and the resulting periodic orbit graph will also be constructive. 
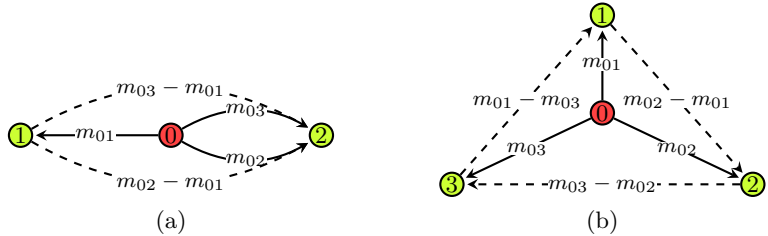
\begin{figure}[h!]
\begin{center}
\subfloat[]{\begin{tikzpicture}[->,>=stealth,shorten >=1pt,auto,node distance=2.8cm,thick, font=\footnotesize] 
\tikzstyle{vertex1}=[circle, draw, fill=couch, inner sep=.5pt, minimum width=3.5pt, font=\footnotesize]; 
\tikzstyle{vertex2}=[circle, draw, fill=melon, inner sep=.5pt, minimum width=3.5pt, font=\footnotesize]; 
\tikzstyle{voltage} = [fill=white, inner sep = 0pt,  font=\scriptsize, anchor=center];

	\node[vertex1] (1) at (-2,0)  {$1$};
	\node[vertex1] (2) at (2,0) {$2$};
	\node[vertex2] (0) at (0, 0) {$0$};

	\draw[thick] (0) -- node[voltage] {$\bm_{01}$} (1);
	
	\draw[thick] (0) edge  [bend right]  node[voltage] {$\bm_{02}$} (2);
	\draw[thick] (0) edge  [bend left] node[voltage] {$\bm_{03}$} (2);
\draw[thick, dashed] (1) edge [bend right]  node[voltage] {$\bm_{02} - \bm_{01}$} (2);
\draw[thick, dashed] (1) edge [bend left]  node[voltage] {$\bm_{03} - \bm_{01}$} (2);
%\draw[thick] (2) --  node[voltage] {$\bm_{03} - \bm_{02}$} (4);
%\draw[thick] (4) --  node[voltage] {$\bm_{01} - \bm_{03}$} (1);
	
\end{tikzpicture}}\hspace{.5in}
\subfloat[]{\begin{tikzpicture}[->,>=stealth,shorten >=1pt,auto,node distance=2.8cm,thick, font=\footnotesize] 
\tikzstyle{vertex1}=[circle, draw, fill=couch, inner sep=.5pt, minimum width=3.5pt, font=\footnotesize]; 
\tikzstyle{vertex2}=[circle, draw, fill=melon, inner sep=.5pt, minimum width=3.5pt, font=\footnotesize]; 
\tikzstyle{voltage} = [fill=white, inner sep = 0pt,  font=\scriptsize, anchor=center];

	\node[vertex1] (1) at (0,1.5)  {$1$};
	\node[vertex1] (2) at (2,-.75) {$2$};
	   \node[vertex1] (4) at (-2, -.75) {$3$};
	\node[vertex2] (3) at (0, .2) {$0$};

	\draw[thick] (3) --  node[voltage] {$\bm_{01}$} (1);
	\draw[thick] (3) -- node[voltage] {$\bm_{02}$} (2);
	\draw[thick] (3) -- node[voltage] {$\bm_{03}$} (4);
\draw[ thick, dashed] (1) --  node[voltage] {$\bm_{02} - \bm_{01}$} (2);
\draw[thick, dashed] (2) --  node[voltage] {$\bm_{03} - \bm_{02}$} (4);
\draw[thick, dashed] (4) --  node[voltage] {$\bm_{01} - \bm_{03}$} (1);
	
\end{tikzpicture}}
\caption{If $v_0$ is adjacent to two distinct vertices, the two candidate edges for insertion in a reverse edge split are the dashed edges in (a). If $v_0$ is adjacent to three distinct vertices, the three candidate edges are the dashed edges in (b). \label{fig:schematic}}
\end{center}
\end{figure}
Now suppose that $v_0$ is adjacent to three distinct vertices $v_1, v_2, v_3$. Suppose the directed edge connecting $v_0$ to $v_i$ has gain assignment $\bm_i$. Then the three candidates for reverse edge split are: $\{v_1, v_2; \bm_{02} - \bm_{01}\}$, $ \{v_2, v_3; \bm_{03} - \bm_{02}\}$, and $ \{v_3, v_1; \bm_{01} - \bm_{03}\}$ (see Figure \ref{fig:schematic}(b)).

Our goal is to prove that there is always at least one edge that can be added such that the resulting periodic orbit graph is constructive. 

Let $H$ be the  graph obtained from $G$ by deleting vertex $v_0$ and its incident edges (as in Lemma \ref{lem:ref1}) Let $G_{ij}$ be a subgraph of $H$ containing the vertices $v_i, v_j$, for $i, j \in \{1, 2, 3\}$. 

Such a subgraph could prevent the addition of the the edge $e = \{v_i, v_j; \bm_{0j} - \bm_{0i}\}$ for one of two reasons: 
\begin{enumerate}
\item the subgraph $(V_{ij}, E_{ij} \cup e)$ would be an over-counted subgraph of $G_0$ (that is, $G_{ij}$ is $(2,2)$-tight already) 
\item or adding the candidate edge would induce a $(2,2)$-tight subgraph of $\langle G_0, m_0 \rangle$ that did not have a constructive gain assignment.  
\end{enumerate}
In Case 1, $G_{ij}$ is $(2,2)$-tight, and in Case 2, $G_{ij}$ is $(2,3)$-tight (it cannot contain any $(2,2)$-tight subgraphs, since these would contain a constructive cycle). Therefore, if one of the candidate edges is {\it not} in the span of any $(2,3)$- or $(2,2)$-tight subgraph of $H$, then we can insert the edge (in this case the gain on the edge is not important) and the claim follows. 

Suppose one of the candidate edges is in the span of a $(2,2)$-tight subgraph of $H$, say $G_{12}$. Then by Lemma \ref{lem:ref1}, neither of $G_{23}$ and $G_{13}$ are $(2,2)$-tight, and at most one of $G_{23}$ and $G_{13}$ are $(2,3)$-tight. Therefore we can always add one of the candidate edges.

Finally, if none of the candidate edges are in the span of a $(2,2)$-tight subgraph of $H$, then the only potentially problematic case we have not eliminated through combinatorial arguments alone is the situation in which all three subgraphs $G_{ij}$ are $(2,3)$-tight. In this case, Proposition \ref{prop:twoConstFailure} applies to show that we can always add an edge, which concludes the proof. 
\end{proof}

%%% If v_0 is only adjacent to two vertices
\begin{prop}
Let $\pog$ be a constructive periodic orbit graph, where $v_0$ is a vertex connected to vertices $v_1$ and $v_2$ by three edges: 
\[\{v_0, v_1; m_{01}\}, \{v_0, v_2; m_{02}\}, \{v_0,v_2; m_{03}\}.\] After deleting $v_0$ it is always possible to add one of the edges $\{v_1, v_2; \bm_{02}- \bm_{01}\}$ or $\{v_1, v_2; \bm_{03} - \bm_{01}\}$ so that the resulting periodic orbit graph $\pogp$ is also constructive (see Figure \ref{fig:schematic}(a)).
\label{prop:v0is2Valent}
\end{prop}

\begin{proof}
First notice that we cannot have a subgraph $G^* \subset G$ satisfying $|E^*| = 2|V^*| - 2$,  $v_0 \notin V^*$, and  $v_1, v_2 \in V^*$, since this would mean that after adding $v_0$ and its three incident edges, the resulting graph would be an overcounted subgraph of $G$. Therefore, any subgraph $G^*$ containing $v_1$ and $v_2$ but not $v_0$ must satisfy $|E^*| \leq 2|V^*| - 3$. 

We now address the question of whether it is possible that after adding either of the candidate edges, a subgraph $G^*$ is created with $|E^*| = 2|V^*| - 2$ but that has no constructive cycles.

%Let $\pog$ be the periodic orbit graph resulting from a reverse edge split on $\pogp$, and suppose we have added the edge $e = \{v_1, v_2; \bm_{02}- \bm_{01}\}$. %and let $\pog$ be the resulting periodic orbit graph after the reverse edge split. 
Suppose there exist vertex-induced subgraphs $\langle G_a, \bm_a \rangle$ and $\langle G_b, \bm_b \rangle$ of $\pog$ where $\langle G_a, \bm_a \rangle$ satisfies:
\begin{enumerate}
	\item $|E_a| = 2|V_a| - 3$ 
	\item $v_0 \notin V_a$
	\item $v_1, v_2 \in V_a$ %(and therefore $e \in E_a$)
	\item all paths from $v_1$ to $v_2$ have net gain $\bm_{02} - \bm_{01}$,
\end{enumerate}
and $\langle G_b, \bm_b \rangle \subset \pog$ satisfies:
\begin{enumerate}
	\item[1'.] $|E_b| = 2|V_b| - 3$ 
	\item[2'.] $v_0 \notin V_b$
	\item[3'.] $v_1, v_2 \in V_b$
	\item[4'.] all paths from $v_1$ to $v_2$ have net gain $\bm_{03} - \bm_{01}$.
\end{enumerate}

Suppose that $\pogp$ is the periodic orbit graph created from $\pog$ by deleting $v_0$ and its incident edges, and adding the edge $e = \{v_1, v_2; \bm_{02}- \bm_{01}\}$.
Let $\langle G_a^*, \bm_a^* \rangle$ and $\langle G_b^*, \bm_b^* \rangle$ be the periodic orbit graphs created from $\langle G_a, \bm_a \rangle$ and $\langle G_b, \bm_b \rangle$ respectively by adding the edge $e = \{v_1, v_2; \bm_{02}- \bm_{01}\}$. Then $\langle G_a^*, \bm_a^* \rangle$ and $\langle G_b^*, \bm_b^* \rangle$ are vertex-induced subgraphs of $\pogp$. 

Then $\langle G_a^*, \bm_a^* \rangle$ is a $(2,2)$-tight subgraph of $\pog$ with no constructive cycles. However, since $G_a^*$ and $G_b^*$ are both $(2,2)$-tight, the intersection of these graphs is also $(2,2)$-tight, and contains $v_1$ and $v_2$ (and therefore contains $e$). Since every $(2,2)$-tight graph is 2-edge connected (it admits a decomposition into two spanning trees), there is some path from $v_1$ to $v_2$ that is distinct from $e$. Because this path is in $\langle G_a^*, \bm_a^* \rangle$, it must have net gain $\bm_{02} - \bm_{01}$. But because the path is also in $\langle G_b^*, \bm_b^* \rangle$, it must also have net gain $\bm_{03} - \bm_{01}$. This is only possible if $\bm_{03} = \bm_{02}$, which contradicts the fact that the original gain assignment $\bm$ is constructive. 

Therefore, subgraphs $\langle G_a, \bm_a \rangle$ and $\langle G_b, \bm_b \rangle$ cannot both exist, and hence it is always possible to add one of the candidate edges. %
\end{proof}

\begin{prop}
Let $\pog$ be a constructive periodic orbit graph. Let $v_0$ be a three-valent vertex incident to the edges $\{v_0, v_1; \bm_{01}\}, \{v_0, v_2; \bm_{02}\}$ and $\{v_0, v_3; \bm_{03}\}$. After deleting $v_0$ it is always possible to add one of the edges $\{v_1, v_2; \bm_{02}- \bm_{01}\}$, $\{v_2, v_3; \bm_{03} - \bm_{02}\}$  or $\{v_3, v_1; \bm_{01} - \bm_{03}\}$ so that the resulting periodic orbit graph $\pogp$ is also constructive (see Figure \ref{fig:schematic}(b)).

\label{prop:twoConstFailure}
\end{prop}

\begin{proof}
We will show that if $\pog$ is a constructive periodic orbit graph with the set-up described above, then there are at most two distinct pairs of vertices from the set $\{v_1, v_2, v_3\}$ that are contained in vertex-induced subgraphs $G_{ij} \subset G$, $i,j \in \{1, 2, 3\}$ satisfying the following (see Figure \ref{fig:twoConstFailureA}):
\begin{enumerate}[(i)]
\item $v_0 \notin V_{ij}$
\item $G_{ij}$ is $(2,3)$-tight
\item $\langle G_{ij}, m|_{G_{ij}} \rangle$ contains no constructive cycle 
\item every path through $\langle G_{ij}, m|_{G_{ij}} \rangle$ originating at $v_i$ and terminating at $v_j$, $i, j \in \{1, 2, 3\}$ has net gain $\bm_{0j} - \bm_{0i}$. 
\end{enumerate}
If three such subgraphs exist, we will not be able to add any of the candidate edges, since each edge edge would create a $(2,2)$-tight subgraph with no constructive cycle. We now show that three subgraphs satisfying (i) -- (iv) cannot exist. 

Toward a contradiction, suppose that there are three such graphs $G_{12}, G_{23}, G_{31}$, with $v_i, v_j \in V_{ij}$. It will be presently be shown that the union of these graphs, $G' = G_{12} \cup G_{23} \cup G_{31} \subset G$ will always satisfy:
\begin{enumerate}[(a)]
\item  $G'$ is a $(2,3)$-tight vertex-induced subgraph of $G$
\item $\langle G', m' \rangle$ contains no constructive cycle
\item  every path through $\langle G', m' \rangle$ originating at $v_i$ and terminating at $v_j$ has net gain $\bm_{0j} - \bm_{0i}$. 
\end{enumerate}

If this is the case, let $V_0 = V' \cup \{v_0\}$, and consider the graph $G_0 = (V_0, E_0)$, where $E_0$ is $E'$ augmented by the three edges connecting $v_0$ with $v_1, v_2, v_3$. Let $\langle G_0, m_0 \rangle$ be the corresponding induced periodic orbit framework (see Figure \ref{fig:twoConstFailureA}). Then $|E_0| = 2|V_0| - 2$, and hence this graph must be constructive. But we know that $\langle G', m' \rangle$ contains no constructive cycle, which means that the constructive cycle in $\langle G_0, m_0 \rangle$ must pass through $v_0$. Hence it must contain two of the edges incident to $v_0$. But any such cycle will have net gain zero, a contradiction. 

\begin{figure}
\begin{center}
\begin{tikzpicture}[very thick,scale=1,>=stealth,->,shorten >=2pt,looseness=0.5,auto] 

\tikzstyle{vertex1}=[circle, draw, fill=couch, inner sep=1pt, minimum width=5pt]; 
\tikzstyle{vertex2}=[circle, draw, fill=melon, inner sep=1pt, minimum width=5pt]; 
\tikzstyle{gain} = [fill=white, inner sep = 0pt,  font=\footnotesize, anchor=center];
\tikzstyle{arrow}=[->,shorten >=1pt,>=stealthÕ,semithick]  
\pgfsetfillopacity{0.3} 
%\pgfsetdash{.1cm}
\fill[cloud] (90:1.5cm) ellipse (25mm and 11mm); 
\draw[very thin] (90:1.5cm) ellipse (25mm and 11mm); 
\fill[cloud, rotate=120] (90:1.5cm) ellipse (25mm and 11mm); 
\draw[ very thin, rotate=120] (90:1.5cm) ellipse (25mm and 11mm); 
\fill[cloud, rotate=240] (90:1.5cm) ellipse (25mm and 11mm); 
\draw[very thin, rotate=240] (90:1.5cm) ellipse (25mm and 11mm); 
\pgfsetfillopacity{1}

	\node[vertex1] (x1) at (-90:2cm){$1$};
	\node[vertex1] (x3) at (-210:2cm)  { $3$};
	\node[vertex1] (x2) at (30:2cm) {$2$};
	\node[vertex1] (x) at (0,0) {$0$};

	\draw[->] (x) -- node[gain] {  $ \bm_{01}$} (x1) 	;
	\draw[->] (x) --  node[gain] {  $ \bm_{02}$} (x2);
	\draw[->] (x) --  node[gain] {  $ \bm_{03}$} (x3);

\draw[->, very thick, bluey,shorten >=15pt] (x1) to [controls=+(-15:1) and +(-45:1)] node[gain] {  $\bm_{02} -  \bm_{01}$} (x2);
%\draw[->, thick, bluey,shorten >=15pt] (x1) to [controls=+(-30:2) and +(-30:2)] node[gain] {  $\bm_{02} -  \bm_{01}$} (x2);
%\draw[->, thick, bluey,shorten >=15pt] (x1) to [controls=+(-30:2.5) and +(-30:2.5)] (x2); 
\node at (2, -2.2) {$G_{12}$};

\draw[->, very thick, melon] (x2) to [controls=+(120:1) and +(60:1)] node[gain]  {  $\bm_{03} -  \bm_{02}$} (x3);
%\draw[->, thick, melon] (x2) to [controls=+(105:2) and +(75:2)] node[gain] {  $\bm_{03} -  \bm_{02}$} (x3);
%\draw[->, thick, melon] (x2) to [controls=+(105:2.5) and +(75:2.5)] (x3) 	; 	
\node at (1.5, 2) {$G_{23}$};
	
\draw[->, very thick, lips] (x3) to [controls=+(225:1) and +(-160:1)] node[gain]  {  $\bm_{01} -  \bm_{03}$} (x1);
%\draw[->, thick, lips] (x3) to [controls=+(210:2) and +(-145:2)] node[gain] {  $\bm_{01} -  \bm_{03}$} (x1);
%\draw[->, thick, lips] (x3) to [controls=+(210:2.5) and +(-145:2.5)] (x1) ; 	
\node at (-2, -2.2) {$G_{31}$};	

\end{tikzpicture}
\caption{Subgraphs $G_{ij}$ of $G$ with the properties (i) -- (iv) of the proof of Proposition \ref{prop:twoConstFailure}. \label{fig:twoConstFailureA}}
\end{center}
\end{figure}
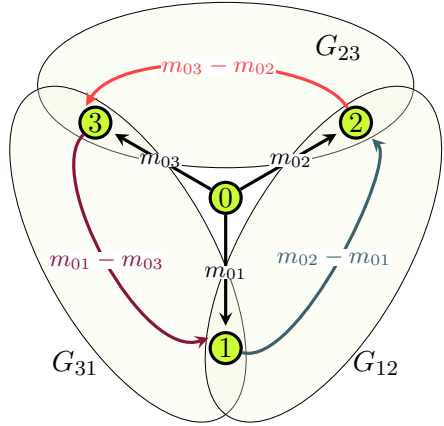

We now show that $\langle G', m' \rangle$ always satisfies properties (a) -- (c) above, and we do this in two cases:

\noindent{\it \underline{Case 1.}} $V_{ij} \cap V_{jk} = \{v_j\}$ for $j \in \{1,2,3\}$\\
In other words, each pair of subgraphs intersects in a single vertex. Here 
\begin{eqnarray*}
|E_{12} \cup E_{23} \cup E_{31}| & = & |E_{12}| + |E_{23}| + |E_{31}|\\
& = & 2(|V_{12}| + |V_{23}| + |V_{31}|) - 9\\
& = & 2|V_{12} \cup V_{23} \cup V_{31}| - 3
\end{eqnarray*}
since 
\begin{eqnarray*}
|V_{12} \cup V_{23} \cup V_{31}| & = & |V_{12}| + |V_{23}| + |V_{31}| - |V_{12} \cap V_{23}| - |V_{23} \cap V_{31}| \\
& & \ \ \ \ - |V_{31} \cap V_{12}| + 2| V_{12} \cap V_{23} \cap V_{31}|\\
& = &  |V_{12}| + |V_{23}| + |V_{31}| - 3.
\end{eqnarray*}
If $\langle G', m' \rangle$ contains a constructive cycle, then it must pass through $v_1, v_2$ and $v_3$. We write the cycle as follows, where the text above the arrow connecting $v_i$ to $v_j$ indicates the net gain on the path from $v_i$ to $v_j$ in the cycle. 
\[v_1 \xrightarrow{m_{02} - m_{01}} v_2 \xrightarrow{m_{03} - m_{02}} v_3 \xrightarrow{m_{01} - m_{03}} v_1.\]
Summing the gains on each part of the cycle we see that it has net gain $0$. Therefore $G'$ satisfies (c) in this case. 

\noindent{\it \underline{Case 2.}} $|V_{ij} \cap V_{jk}| > 1$ for at least one $j \in \{1,2,3\}$. \\
By a repeated application of Lemma \ref{lem:intersection}, we find that the union of these three graphs satisfies $|E'| = 2|V'| - 3$. (Let $G^* = G_{12} \cup G_{23}$. Assuming that $|V_{12} \cap V_{23}| > 1$, apply Lemma \ref{lem:intersection} to see that $|E^*| = 2|V^*| - 3$. Now it must be the case that $|V^* \cap V_{31}| > 1$ as well, since $v_1, v_3$ are in both vertex sets. Another application of Lemma \ref{lem:intersection} gives the result.) Note further that the intersection of $G^*$ and $G_{31}$ contains at least two vertices ($v_1$ and $v_3$), and satisfies $|E^* \cap E_{31}| = 2|V^* \cap V_{31}| -3$ by Lemma \ref{lem:intersection}. Furthermore, this intersection is non-empty. Equivalently, the intersection $V_{12} \cap V_{23} \cap V_{31}$ is non-empty (see Figure \ref{fig:twoConstFailure}).
% picture for the proof that we can't have three failures of the constructive topology
\begin{verse}
\begin{figure}
\begin{center}
\begin{tikzpicture}[thick,scale=1,>=stealth,->,shorten >=2pt,looseness=0.5,auto] 

\tikzstyle{vertex1}=[circle, draw, fill=couch, inner sep=1pt, minimum width=5pt]; 
\tikzstyle{vertex2}=[circle, draw, fill=melon, inner sep=1pt, minimum width=5pt]; 
\tikzstyle{arrow}=[->,shorten >=1pt,>=stealthÕ,semithick]  
\tikzstyle{gain} = [fill=white, inner sep = 0pt,  font=\footnotesize, anchor=center];

\pgfsetfillopacity{0.3} 
%\pgfsetdash{.1cm}
\fill[cloud] (90:1.5cm) circle (20mm); 
\draw[very thin] (90:1.5cm) circle (2cm);
\fill[cloud] (210:1.5cm) circle (20mm); 
\draw[very thin] (210:1.5cm) circle (20mm); 
\fill[cloud] (-30:1.5cm) circle (20mm); 
\draw[very thin] (-30:1.5cm) circle (20mm);
\pgfsetfillopacity{1}

	\node[vertex1] (x1) at (-90:1.5cm){$1$};
	\node[vertex1] (x3) at (-210:1.5cm)  { $3$};
	\node[vertex1] (x2) at (30:1.5cm) {$2$};
	\node[vertex1] (x) at (0,0) {$x$};
	
	\draw[->] (x) --  (x1) 	node [left,text centered,midway]{};% {  $ \bm_{01}$};
	\draw[->] (x) -- (x2) node [above,text centered,midway]{};% {  $ \bm_{02}$};;
	\draw[->] (x) -- (x3) node [above , text centered,midway]{};% {  $ \bm_{03}$};;

%\draw[->, thick, bluey,shorten >=15pt] (x1) to [controls=+(-15:1) and +(-45:1)] (x2);
\draw[->, very thick, bluey,shorten >=15pt] (x1) to [controls=+(-30:2) and +(-30:2)] node[gain]  {  $\bm_{02} -  \bm_{01}$} (x2);
%\draw[->, thick, bluey,shorten >=15pt] (x1) to [controls=+(-30:2.5) and +(-30:2.5)] (x2) 	node [ right, midway] {  $\bm_{02} -  \bm_{01}$}; 
\node at (2, -2.2) {$G_{12}$};

%\draw[->, thick, melon] (x2) to [controls=+(120:1) and +(60:1)] (x3);
\draw[->, very thick, melon] (x2) to [controls=+(105:2) and +(75:2)] node[gain]  {  $\bm_{03} -  \bm_{02}$} (x3);
%\draw[->, thick, melon] (x2) to [controls=+(105:2.5) and +(75:2.5)] (x3) 	node [above, midway] {  $\bm_{03} -  \bm_{02}$}; 	
\node at (1.5, 2) {$G_{23}$};
	
%\draw[->, thick, lips] (x3) to [controls=+(225:1) and +(-160:1)] (x1);
\draw[->, very thick, lips] (x3) to [controls=+(210:2) and +(-145:2)] node[gain]  {  $\bm_{01} -  \bm_{03}$} (x1);
%\draw[->, thick, lips] (x3) to [controls=+(210:2.5) and +(-145:2.5)] (x1) 	node [left, midway] {  $\bm_{01} -  \bm_{03}$}; 	
\node at (-2, -2.2) {$G_{31}$};	

\end{tikzpicture}
\end{center}
\caption{Case 2 of the proof of Proposition \ref{prop:twoConstFailure}. \label{fig:twoConstFailure}} 
\end{figure}
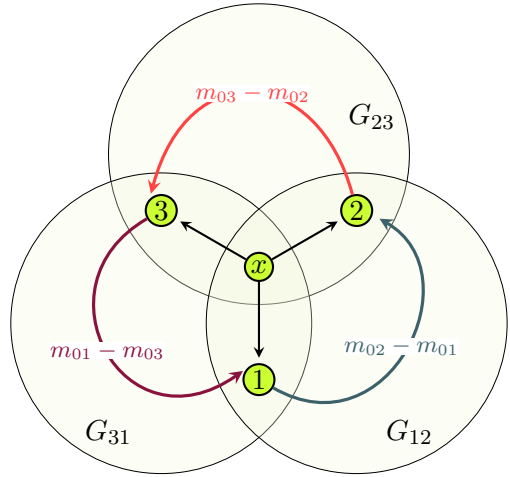\end{verse}

We now demonstrate that $\langle G', m' \rangle$ contains no constructive cycle. We assume that there is a constructive cycle in $\langle G', m' \rangle$, and we will obtain a contradiction to condition {\it (iii)}. We do this in two parts, first by showing that there are no constructive cycles in the union of any pair of subgraphs (a), and next showing that there there are no constructive cycles in the union of all three (b). 

{\it \underline{Case 2a.}}  \ \ Suppose that there is a constructive cycle in the periodic orbit graph induced by $(V_{12} \cup V_{23}, E_{12} \cup E_{23})$.  Suppose that  $|V_{12} \cap V_{23}| > 1$, and that the constructive cycle passes through vertices $x$ and $y$, where $x, y \in V_{12} \cap V_{23}$. The simplest case is pictured in  Figure \ref{fig:twoGraphsConstFailure}.

% picture for the proof that we can't have three failures of the constructive topology
\begin{verse}
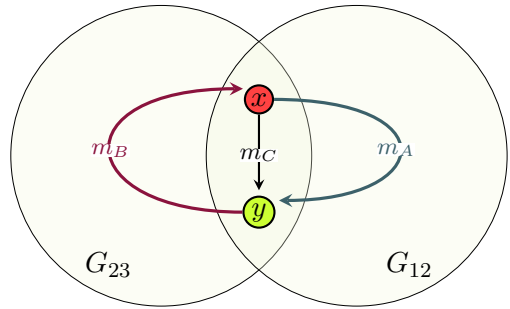
\begin{figure}
\begin{center}
%\begin{tikzpicture}[very thick,scale=1,>=stealth,->,shorten >=2pt,looseness=0.5,auto] 
\begin{tikzpicture}[thick,scale=1,>=stealth,->,shorten >=2pt,looseness=0.5,auto]

\tikzstyle{vertex1}=[circle, draw, fill=couch, inner sep=1pt, minimum width=5pt]; 
\tikzstyle{vertex2}=[circle, draw, fill=melon, inner sep=1pt, minimum width=5pt]; 
\tikzstyle{arrow}=[->,shorten >=1pt,>=stealthÕ,semithick]  
\tikzstyle{gain} = [fill=white, inner sep = 0pt,  font=\footnotesize, anchor=center];

\pgfsetfillopacity{0.3} 
%\pgfsetdash{.1cm}
%\fill[cloud] (90:1.5cm) circle (20mm); 
%\draw[dashed, very thin] (90:1.5cm) circle (2cm);
\fill[cloud] (210:1.5cm) circle (20mm); 
\draw[very thin] (210:1.5cm) circle (20mm); 
\fill[cloud] (-30:1.5cm) circle (20mm); 
\draw[very thin] (-30:1.5cm) circle (20mm);
\pgfsetfillopacity{1}

	\node[vertex1] (y) at (-90:1.5cm){$y$};
	\node[vertex2] (x) at (0,0) {$x$};
	
	\draw[->] (x) --  node[gain]  {  $ \bm_{C}$} (y) 	;
	%\draw[->] (x) -- (x2) node [above,text centered,midway] {  $ \bm_{02}$};;
	%\draw[->] (x) -- (x3) node [above , text centered,midway] {  $ \bm_{03}$};;

%\draw[->, thick, bluey,shorten >=15pt] (y) to [controls=+(-15:1) and +(-45:1)] (x);
%\draw[->, thick, bluey,shorten >=15pt] (x1) to [controls=+(-30:2) and +(-30:2)] (x2);
\draw[->, very thick, bluey,shorten >=12pt] (x) to [controls=+(0:2.5) and +(0:2.5)] node[gain] {  $\bm_A$} (y); 
\node at (2, -2.2) {$G_{12}$};

%\draw[->, thick, lips] (y) to [controls=+(225:1) and +(-160:1)] (x);
%\draw[->, thick, lips] (x3) to [controls=+(210:2) and +(-145:2)] (x1);
\draw[->, very thick, lips] (y) to [controls=+(180:2.5) and +(180:2.5)] node[gain] {  $\bm_B$} (x) 	; 	
\node at (-2, -2.2) {$G_{23}$};	

\end{tikzpicture}
\end{center}
\caption{Two subgraphs satisfying {\it (i) -- (iv)} of Proposition \ref{prop:twoConstFailure} whose intersection contains more than one vertex. \label{fig:twoGraphsConstFailure}} 
\end{figure}\end{verse}

% picture for the proof that we can't have three failures of the constructive topology
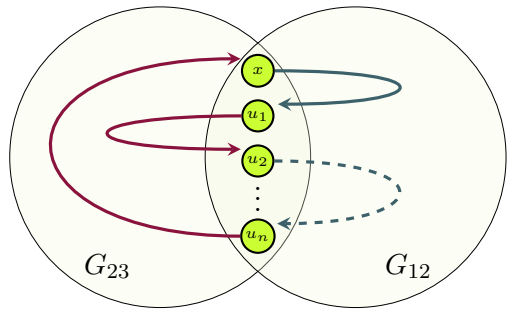
\begin{figure}
\begin{center}
\begin{tikzpicture}[thick,scale=1,>=stealth,->,shorten >=2pt,looseness=0.5,auto] 

\tikzstyle{vertex1}=[circle, draw, fill=couch, inner sep=1pt, minimum width=5pt,font=\tiny]; 
\tikzstyle{vertex2}=[circle, draw, fill=melon, inner sep=1pt, minimum width=5pt]; 
\tikzstyle{arrow}=[->,shorten >=1pt,>=stealthÕ,semithick]  
\pgfsetfillopacity{0.3} 
%\pgfsetdash{.1cm}
%\fill[cloud] (90:1.5cm) circle (20mm); 
%\draw[dashed, very thin] (90:1.5cm) circle (2cm);
\fill[cloud] (210:1.5cm) circle (20mm); 
\draw[very thin] (210:1.5cm) circle (20mm); 
\fill[cloud] (-30:1.5cm) circle (20mm); 
\draw[very thin] (-30:1.5cm) circle (20mm);
\pgfsetfillopacity{1}

	\node[vertex1] (un) at (0, -1.8){$u_n$};
	\node[vertex1] (u2) at (0, -.8){$u_2$};
	\node[vertex1] (u1) at (0, -.2){$u_1$};
	\node[vertex1] (x) at (0,.4) {$\ x \ $};
	\node at (0, -1.2) {$\vdots$};
	
%	\draw[->] (x) --  (y) 	node [left,text centered,midway] {  $ \bm_{C}$};
\draw[->, very thick, bluey,shorten >=12pt] (x) to [controls=+(0:2.5) and +(0:2.5)] (u1);% 	node [ right, midway] {  $\bm_A$}; 
\draw[->, very thick, lips] (u1) to [controls=+(180:2.5) and +(180:2.5)] (u2);% 	node [left, midway] {  $\bm_B$}; 	
%\draw[->, very thick, bluey,shorten >=12pt] (u2) to [controls=+(0:2.5) and +(0:2.5)] (0, -1.3);
\draw[->, very thick, bluey,shorten >=12pt, dashed] (u2) to [controls=+(0:2.5) and +(0:2.5)] (un);
\draw[->, very thick, lips] (un) to [controls=+(180:3.5) and +(180:3.5)] (x);% 	node [left, midway] {  $\bm_B$}; 	

\node at (-2, -2.2) {$G_{23}$};	
\node at (2, -2.2) {$G_{12}$};

\end{tikzpicture}
\end{center}
\caption{A candidate constructive cycle, {\it Case 2a} of the proof of Proposition \ref{prop:twoConstFailure}. \label{fig:twoGraphsConstFailureII}} 
\end{figure}

Suppose first that the constructive cycle in $G_{12} \cup G_{23}$ is as pictured in Figure \ref{fig:twoGraphsConstFailure}, and the cycle does not go in and out of $G_{12} \cap G_{23}$ (as pictured in Figure \ref{fig:twoGraphsConstFailureII}). Denote the part of the constructive cycle from $x$ to $y$ in $G_{12}$ by $x \xrightarrow{m_A} y$. Similarly, let $y \xrightarrow{m_B} x$ denote the part of the constructive cycle from $y$ to $x$ in $G_{23}$.  Then $\bm_A + \bm_B \neq 0$ by assumption.  By Corollary \ref{cor:intersection}, the graph $(V_{12} \cap V_{23}, E_{12} \cap E_{23})$ is connected. Hence there exists a path through this graph that connects $x$ to $y$. Let the net gain of this path be $\bm_C$. Then we write the cycle
$x \xrightarrow{m_A} y \xrightarrow{m_B} x$ as 
\[x \xrightarrow{m_A} y \xrightarrow{m_C} x \xrightarrow{-m_C} y \xrightarrow{m_B} x.\]
But $m_A +m_C = m_B - m_C = 0$, hence the net gain on this cycle is $0$, which contradicts our assumption that it was constructive.

%
%$$\bm_A - \bm_C = 0$$
%$$\Rightarrow \bm_A = \bm_C$$
%$$\Rightarrow \bm_C + \bm_B \neq 0.$$
%But this is a constructive cycle in $G_{23}$, a contradiction.

Now suppose the constructive cycle in $G_{12} \cup G_{23}$ is as pictured in Figure \ref{fig:twoGraphsConstFailureII}, and the path weaves in and out of the intersection $G_{12} \cap G_{23}$. Let $x, u_1, \dots, u_n \in V_{12} \cap V_{23}$, and suppose the constructive cycle is as follows: 
\[x \xrightarrow{m_1} u_1 \xrightarrow{m_2} u_2 \xrightarrow{m_3} \dots \xrightarrow{m_n} u_n \xrightarrow{m_{n+1}} x,\]
with $\sum m_i \neq 0$, and where the path $x \xrightarrow{m_1} u_1$ is completely contained in $G_{12}$, $u_1 \xrightarrow{m_2} u_2$ is completely contained in $G_{23}$ and so on, with the path segments continuing to alternate between $G_{12}$ and $G_{23}$. 

Expand the constructive cycle as follows, adding a path to $x$ in $G_{12} \cap G_{23}$ between each path segment $u_i \rightarrow u_{i+1}$. Since each such path is traversed in both directions, it does not change the net gain on the cycle.
\[x \xrightarrow{m_1} u_1 \rightarrow x \rightarrow u_1 \xrightarrow{m_2} u_2 \rightarrow x \rightarrow u_2 \xrightarrow{m_3} \dots  \xrightarrow{m_n} u_n \xrightarrow{m_{n+1}} x.\]
The cycle $x \rightarrow u_1 \rightarrow x$, where $u_1 \rightarrow x$ is completely contained in $G_{12}$, and therefore has net gain zero. Similarly, the cycle $x \rightarrow u_1 \rightarrow u_2 \rightarrow x$ is completely contained in $G_{23}$ and therefore has net gain zero. Continuing in this way, we see that the constructive cycle is the sum of cycles with net gain zero, and hence is not constructive:

\[\underbrace{x \xrightarrow{m_1} u_1 \rightarrow}_0 \underbrace{x \rightarrow u_1 \xrightarrow{m_2} u_2 \rightarrow}_0 \underbrace{x \rightarrow u_2 \xrightarrow{m_3} \dots}_0 \underbrace{\dots \xrightarrow{m_n} u_n \xrightarrow{m_{n+1}} x}_0.\]

{\it \underline{Case 2b.}} Now assume that there is a constructive cycle in the subgraph of $G$ on the vertices $V_{12} \cup V_{23} \cup V_{31}$. See Figure \ref{fig:twoConstFailure}.  By a similar argument to the previous case, suppose that the constructive cycle is written as the sum of three paths, one through each of the graphs. That is, let $x_1 \in V_{31} \cap V_{12}$, $x_2 \in V_{12} \cap V_{23}$, and $x_3 \in V_{23} \cap V_{31}$. If any of the vertices $x_1, x_2, x_3$ is in the intersection of all three graphs, then we are in the situation described in Case 2a. So we assume that this is not the case, and the constructive cycle does not pass through any vertex of the intersection. In the simplest case, the the constructive cycle may be broken into three components, one in each subgraph $G_{ij}$. We write 
\[x_1 \xrightarrow{m_{02} - m_{01}} x_2 \xrightarrow{m_{03} - m_{02}} x_3 \xrightarrow{m_{01} - m_{03}} x_1.\]
Summing the gains on each component of the cycle it is clear that the cycle has net gain $0$, and is therefore not constructive (Figure \ref{fig:twoConstFailure}).

In the case that the constructive cycle cannot be broken into these three pieces, we use the same approach as in Case 2a. Now the constructive cycle weaves in and out of the subgraphs $G_{ij}$. 
Since the intersection is non-empty, let $x \in V_{12} \cap V_{23} \cap V_{31}$. Each pairwise intersection $V_{ij} \cap V_{jk}$ is connected, hence for the vertex $x_i \in V_{ki} \cap V_{ij}$ there is a path connecting $x$ to $x_i$. As in the previous case, expand the constructive cycle by adding a path to and from the vertex $x$ until the cycle is a sum of smaller cycles, each of which is completely contained in $G_{ij}$ for some $i, j$. As before we see that the original cycle hence has net gain $0$. 
%Let this path have net gain $\bm_i$. Similarly we have paths connecting vertices $x_j$ and $x_k$ respectively to the vertex $x$. Then 

%
%$x_1 A x_2 B x_3 C x_1$, where $A \in G_{12}, B \in G_{23}, C \in G_{31}$. Let these paths have cycle gains $\bm_A, \bm_B, \bm_C$ respectively, and our assumption is that $\bm_A + \bm_B + \bm_C \neq 0$.  
%

%$$ \bm_A - \bm_{01} + \bm_{03} = 0$$
%$$ \bm_B - \bm_{02} + \bm_{01} = 0$$
%$$ \bm_C - \bm_{03} + \bm_{02} = 0.$$
%But summing these three expressions gives $\bm_A + \bm_B + \bm_C = 0$, which contradicts our assumption. As in the previous case, the union of the three graphs can have no constructive cycle. 

To see that $\langle G', m' \rangle$ also satisfies property (c), we consider without loss of generality, all paths $P$ from $v_1$ to $v_2$ through $G'$. If each vertex of the path is in $V_{12}$ then it has net gain $\bm_{02} - \bm_{01}$ by hypothesis. If some vertex in $P$ is not in $V_{12}$, then suppose $P$ has net gain $\bm_P$. Then $\bm_P - (\bm_{02} - \bm_{01}) = 0$, since $G'$ has no trivial cycles, by (b). Hence $\bm_P = \bm_{02} - \bm_{01}$, as desired. 

%If this is the case, then there is a subgraph $G' \subset G$ that contains $v_1, v_2, v_3$ but not $v_0$, and satisfies $|E'| = 2|V'| - 3$. Furthermore, this graph contains no constructive cycle, and all paths connecting $v_i$ to $v_j$ have net gain $\bm_{0j} - \bm_{0i}$. Let $V_0 = V' \cup \{v_0\}$, and consider the graph $G_0 = (V_0, E_0)$. $E_0$ will be $E'$ augmented by the three edges connecting $v_0$ with $v_1, v_2, v_3$. Then $|E_0| = 2|V_0| - 2$, and hence this graph must be constructive. But we know that $G'$ contains no constructive cycle, which means that the constructive cycle in $G_0$ must pass through $v_0$. Hence it must contain two of the edges incident to $v_0$. But any such cycle will have net gain zero, a contradiction. 
\end{proof}

%%% FURTHER WORK
\section{Further work and related questions}
\label{sec:conclusion}

\subsection{Algorithms}
An algorithm for determining the rigidity of a periodic orbit framework on $\Tor^2$ appears in \cite{fixedConeAlgorithms}, with running time $O(n^3)$. The basic idea is based on the pebble game algorithm for finite frameworks due to Jacobs and Hendrickson \cite{pebbleGame}, and developed in \cite{pebbleGameSparse, SparsityCertifying}. The key idea for the fixed torus algorithm in \cite{fixedConeAlgorithms} is to run the $(2,3)$- and $(2,2)$-pebble games simultaneously. %A similar solution to this problem was independently considered by the author in section 6.3 of \cite{myThesis}, but with a slower running time. 

\subsection{Higher dimensions}
In \cite{ThesisPaper1} we presented necessary conditions for rigidity on the $d$-dimensional fixed torus $\Tor^d$. Unfortunately, finding sufficient conditions for generic rigidity on the $d$-dimensional fixed torus rests on solving {\it finite} $d$-dimensional rigidity (finite rigidity is combinatorially characterized for $d = 1, 2$ but not for higher dimensions). For example, it is possible to embed the well-known ``double bananas" example in a three-dimensional periodic framework. See \cite{ThesisPaper1} for further details. 

\subsection{Body-bar frameworks on the fixed torus}
In contrast to the situation for bar-joint frameworks, the generic rigidity of body-bar frameworks on the fixed torus has recently been completely described. The characterization is based on a sparsity condition which depends on the dimension of the gain space. Let $H$ be a multigraph, possibly including loops, with vertex and edge sets $V(H)$ and $E(H)$ respectively. Let $Y \subset E(H)$, and let $V(Y)$ be the set of vertices incident to the edges of $Y$. We use $\gs(Y)$ to denote the gain space of the subgraph of $\bbog$ generated by $(V(Y), Y)$.
\begin{thm}
$\bbog$ is a periodic orbit graph corresponding to a generically minimally rigid body-bar periodic framework in $\mathbb R^d$ if and only if  $|E(H)| = {d+1 \choose 2}|V(H)| - d$ and for all non-empty subsets $Y \subset E(H)$ of edges
		\begin{equation*}
			|Y| \leq {d+1 \choose 2}|V(Y)| - {d+1 \choose 2} + \sum_{i=1}^{|\gs(Y)|}(d-i).
		\end{equation*}	
\end{thm}
The basic idea of this result is that as the dimension of the gain space increases, the maximum number of edges which may be independent also increases. There is an inductive proof of this result for $d\leq 3$ \cite{bodyBar} ($d=1, 2$ follow from the bar-joint characterizations). A non-inductive proof for all dimensions was recently announced in \cite{matroidsOfGainGraphs}, as part of a more general set of results about body-bar frameworks with point group symmetries. A general theory of periodic body-bar frameworks on the flexible torus has been set out in \cite{PeriodicBarBody}, but without a characterization of rigidity based on the underlying gain graph.

\subsection{Inductive constructions on the flexible torus}
In \cite{myThesis}, a characterization was established of the generic rigidity of periodic frameworks on a partially variable torus (allowing one degree of flexibility). Together with Anthony Nixon, we outlined an inductive proof of this result \cite{NixonRoss}. 

\begin{thm}[Nixon and Ross \cite{NixonRoss}]
A framework $\pofw$ is generically minimally rigid on the partially variable torus (with one degree of freedom) if and only if it can be constructed from a single loop by a sequence of extended gain-preserving Henneberg operations. 
\label{thm:partFlexHenn}
\end{thm}

The operations referred to in the theorem above contain the gain-preserving vertex addition and edge split operations described in this paper, but we also require an additional move to deal with a special class of graphs for which the existing moves are insufficient. Further challenges arise when attempting to apply inductive techniques to the fully variable torus (having three degrees of freedom, this is what is called a ``periodic framework" in \cite{periodicFrameworksAndFlexibility}), due to the fact that there may no longer be vertices of degree 2 or 3, necessitating the development of further inductive moves. %Of course a complete combinatorial description of frameworks on the fully variable torus exists \cite{Theran}, but we believe that an inductive characterization would be a rich addition to our knowledge of these structures. \\

\subsection{Theta graphs}
A property that emerges in the proof of Proposition \ref{prop:twoConstFailure}, is the {\it theta graph property} \cite{BiasedGraphsI}. A {\it theta graph} is a subdivision of the triple link graph (two vertices connected by three internally disjoint paths). The proof of Proposition \ref{prop:twoConstFailure} established that whenever the union of two cycles with net gain zero is a theta graph, the third cycle in the union also has net gain zero. In other words, the cycles with net gain zero form a {\it linear subclass} of the set of all cycles. 

This property forms the basis for the theory of biased matroids in \cite{BiasedGraphsI}. It is known that the balanced cycles (cycles having net gain zero) of any gain graph are a linear class. It is thus natural to ask whether gain graphs with other group labels also admit Henneberg-type constructions. This question has been considered in \cite{gainSparsity} for frameworks in the plane with cyclic or odd-order dihedral symmetry.

\noindent {\it Acknowledgements.} The author wishes to thank Walter Whiteley, and the anonymous referees for numerous helpful suggestions on earlier versions of this material.

\bibliographystyle{abbrv} 
\bibliography{Papers} 

\end{document}